\documentclass[final,onefignum,onetabnum,leqno]{siamart171218}

\usepackage{amsfonts}
\usepackage{amsmath}
\usepackage{amssymb}
\usepackage{amsopn}
\usepackage{graphicx}
\usepackage{epstopdf}
\usepackage{algpseudocode}
\usepackage{xfrac}
\usepackage{xcolor}
\usepackage{nicefrac}
\usepackage{hyperref}

\def\bA{\boldsymbol{A}}
\def\baj{\boldsymbol{a}_j}
\def\bb{\boldsymbol{b}}
\def\by{\boldsymbol{y}}
\def\bup{\widetilde{\boldsymbol{x}}}

\def\bD{\boldsymbol{D}}
\def\be{\boldsymbol{e}}
\def\bf{\boldsymbol{f}}

\def\zero{\boldsymbol{0}}
\def\ex{^\star}
\def\bG{\boldsymbol{G}}
\def\bI{\boldsymbol{I}}
\def\bL{\boldsymbol{L}}
\def\bJ{\boldsymbol{J}}
\def\bM{\boldsymbol{M}}

\def\bV{\boldsymbol{V}}
\def\prederr{\boldsymbol{p}}

\def\tr{\mathrm{tr}}

\def\bV{\boldsymbol{V}}
\def\bv{\boldsymbol{v}}
\def\bx{\boldsymbol{x}}
\def\bxu{\widetilde{\boldsymbol{x}}}
\def\bbx{\widehat{\boldsymbol{x}}}

\def\sd{\boldsymbol{s}}
\def\sdt{\widetilde{\boldsymbol{s}}}
\def\Kdown{\mathcal{K}^\downarrow}
\def\Kup{\mathcal{K}^\uparrow}
\def\TA{Twin Algorithm}
\def\MSA{Mutual-Step Algorithm}
\def\KO{Kaczmarz+Oracle}
\def\TM{Twin Method}
\def\MSM{Mutual-Step Method}
\def\std{\sigma}

\def\code{\textsf} 

\newtheorem{keypoint}{Key Point}

\algnewcommand\algorithmicoutput{\textbf{Output:}}
\algnewcommand\Output{\item[\algorithmicoutput]}
\algnewcommand\algorithmicoptional{\textbf{Optional:}}
\algnewcommand\Optional{\item[\algorithmicoptional]}
\algnewcommand{\IIf}[1]{\State\algorithmicif\ #1\ \algorithmicthen}
\algnewcommand{\EndIIf}{\unskip\ \algorithmicend\ \algorithmicif}

\newcommand{\ph}{\phantom}

\newcommand{\smtxa}[2]{
	{\mbox{\scriptsize
			$\left[\!\!
			\begin{array}{#1}
			#2
			\end{array} \!\! \right]$}}}

\ifpdf
\DeclareGraphicsExtensions{.eps,.pdf,.png,.jpg}
\else
\DeclareGraphicsExtensions{.eps}
\fi

\newsiamremark{remark}{Remark}
\newsiamthm{assumptions}{Assumption}

\headers{A twin error gauge for Kaczmarz's iterations}{Van Lith, Hansen, and Hochstenbach}

\title{{A twin error gauge for Kaczmarz's iterations}\thanks{Version \today.
}}

\author{B.~S.~van Lith\thanks{Department of Applied Mathematics and Computer Science, Technical University of Denmark,
		DK-2800 Kgs.~Lyngby, Denmark (\email{bavli@dtu.dk, pcha@dtu.dk}).}
	\and P.~C.~Hansen\footnotemark[2]
	\and M.~E.~Hochstenbach\thanks{Department of Mathematics and Computer Science, TU Eindhoven, \url{http://www.win.tue.nl/\~hochsten/}.}}

\begin{document}
	\maketitle
	
	
	\begin{abstract}
		We propose two new algebraic reconstruction techniques based on Kaczmarz's method that produce a regularized solution to noisy tomography problems. Tomography problems exhibit semi-convergence when iterative methods are employed, and the aim is therefore to stop near the semi-convergence point. Our approach is based on an error gauge that is constructed by pairing standard down-sweep Kaczmarz's method with its up-sweep version; we stop the iterations when this error gauge is minimal.
		The reconstructions of the new methods differ from standard Kaczmarz iterates in that our final result is the average of the stopped up- and down-sweeps.
		Even when Kaczmarz's method is supplied with an oracle that provides the exact error--and is therefore able to stop at the best possible iterate --
		our methods have a lower two-norm error in the vast majority of our test cases.
		In terms of computational cost, our methods are a little cheaper than standard Kaczmarz equipped with a statistical stopping rule.
	\end{abstract}
	
	

	\section{Introduction}
	\label{sec:intro}
	The image reconstruction problem in X-ray tomography can be formulated
	as a large, sparse linear system of equations, i.e.,
	\begin{equation}
		\label{eq:Axb}
		\bA \, \bx = \bb, \qquad
		\bA \in \mathbb{R}^{m \times n} , \quad \bb \in \mathbb{R}^m , \quad
		\bx \in \mathbb{R}^n .
	\end{equation}
	Here, the vectors $\bb$ and $\bx$ represent the measured data (the sinogram) and the image to be reconstructed, respectively.
	The system matrix $\bA$ represents a discretization of the forward problem \cite{discretization};
	there are no restrictions on its dimensions $m$ and $n$.
	There is inherently some noise present in the data $\bb$, and one of the
	key challenges in tomographic reconstruction is to compute a good
	reconstruction in the presence of these errors.
	
	Tomographic reconstruction problems are a type of inverse problems where the
	forward operator, in the continuous formulation, is a smoothing operation known as the Radon transform in the case of 2D parallel-beam scanning. The continuous problem is mildly ill-posed \cite{bertero},
	which leads to a poorly conditioned matrix~$\bA$, especially when the system is large.

	The system \eqref{eq:Axb} is usually too large to solve by factorization methods, and iterative linear solvers are used.
	These solvers exhibit \emph{semi-convergence}
	\cite{natterer} in the presence of noise, meaning that initially the
	reconstruction error decreases but eventually it increases.
	The error consists of two parts, the iteration error and the noise error.
	The iteration error decreases steadily, and in the case of error-free data the classical
	asymptotic convergence theory applies. Initially the noise error is small but it steadily increases
	until the iterative method has ``inverted the noisy data'' rather than the clean data;
	see, e.g., \cite{elfving_2012,elfving_2014} for more details.
	
	To obtain meaningful solutions to noisy problems we need
	to stop the iterations at the semi-convergence point where the
	reconstruction error is at a minimum.
	The iteration number therefore acts as a regularization parameter.
	There are various methods that estimate a good optimal regularization parameter
	and these can be used as stopping rules for the iterations
	\cite{bardsley,Reichel2013}.

	Many of the parameter-choice/stopping rules are based on statistical
	properties of the noise, for instance when using generalized cross-validation
	\cite{wahba} and unbiased predictive risk estimation \cite[Sec.~7.1]{vogel}
	and when using variations of the discrepancy principle
	\cite[Sec.~7.2]{Hansen1998}. 
	
	All these stopping rules may work well for simultaneous iterative reconstruction techniques
	such as Landweber or Cimmino~\cite{Hansen2018}.
	The reason is that these methods tend to produce error histories
	that are very flat around the minimum.
	A few hundred iterations more or less often does not really matter,
	and the quality of the reconstruction is only little affected.
	This is not so for Kaczmarz's method which tends to converge much faster
	\cite{elfving_2014} and thus has a fairly narrow window of opportunity
	around the minimum of the error history; eager readers see Figure~\ref{fig:effect_of_noise} in Section~\ref{sec:numerical_experiments}.
	
	We propose a completely different approach, one without any statistical assumptions on the noise,
	which is based on error estimation. Using several numerical examples, we show that our stopping rules and algorithms perform very well.
	To illustrate the point, we compare our methods with an oracle that provides
	the true error and is therefore able to stop at the best possible iterate.
	
	In our analysis of the proposed methods, we show that our approach is theoretically sound for consistent systems. In the numerical examples, however, we only consider the noisy case.
	
	The rest of the paper is organized as follows.
	Section~\ref{sec:kaczmarz} introduces the necessary background theory
	for Kaczmarz's method.
	In Section~\ref{sec:error_gauge_TA} we analyze the errors and use this analysis
	to propose a new way to estimate the error.
	The insight then leads to a new algorithm that is presented in Section~\ref{sec:MSA}.
	Finally, in Section~\ref{sec:numerical_experiments} we present numerical examples
	that illustrate our theory and compare our new algorithm with existing ones.
	Throughout this work, we will use the following common notations:
	\begin{itemize}
		\item A column vector is denoted by a bold lower-case character, while a bold upper-case character is a matrix. Normal font is used for scalars.
		
		\item We will  use the two-norm and denote it by $\| \cdot \|$.
		
		\item The expectation operator is written as $\mathbb{E}$.
		
		\item Exact quantities are marked with a superscript $\star$, while objects furnished with a tilde ($\sim$) are related to an alternative (up-sweep) version of Kaczmarz's methods, cf.~Section~\ref{sec:analysis}.
		
		\item The symbol $\gets$ in pseudocode means assignment.
	\end{itemize}

	\section{Background theory}
	\label{sec:kaczmarz}
	
	Here we set the stage by summarizing some basic results
	pertaining to Kaczmarz's method, also known as the algebraic reconstruction technique (ART) \cite{gordon_bender,van_dijke}.
	
	\subsection{Kaczmarz and its convergence}
	Let $k=1,2,3,\dots$ denote full sweeps through the rows of $\bA$, and let
	$\bx_0$ denote the starting vector.
	Let $\omega \in (0,2)$ be a relaxation parameter, let
	$\baj^T$ denote the $j$th row of $\bA$, and let $b_j$ denote the
	$j$th element of $\bb$. For every update $i$, a row index $j$ needs to be chosen. There are various strategies for picking the row index $j$, for instance randomized or cyclic \cite{Popa2018};
	here we consider only cyclic down-sweeps, $j=i$, and up-sweeps, $j=m-i+1$.
	Starting from the common choice $\bx_0 = \zero$, in the $k$th sweep we sequentially perform the updates
	\begin{subequations}\label{eq:ART}\begin{align}
			\bx_k^{(0)} &= \bx_{k-1} , \\
			\bx_k^{(i)} &= \bx_k^{(i-1)} + \omega \, \frac{b_j - \baj^T \bx_k^{(i-1)}}{\|\baj\|^2} \, \baj,
			\qquad i=1,2,\dots,m \ , \\
			\bx_{k} &= \bx_k^{(m)} \ .
	\end{align}\end{subequations}
	If there are rows with zero norm, it is most natural to skip them. This is effectively the same as deleting zero rows from $\bA$ and the corresponding entries from $\bb$. When $\omega = 1$, Kaczmarz's method has a nice geometrical interpretation: each update projects onto the hyperplane represented by the $j$th equation.
	
	For the down-sweep version, an entire sweep \eqref{eq:ART} of all equations in Kaczmarz's method can be written as
	\begin{equation}\label{eq:Elfving_theorem}
		\bx_{k+1} = \bx_k +  \bA^T \bL^{-1} ( \bb - \bA \bx_k), \qquad
		\bL = \code{slt}(\bA\bA^T) + \tfrac{1}{\omega} \bD \ ,
	\end{equation}
	where $\code{slt}(\cdot)$ extracts the strictly lower triangular part
	and $\bD = \mathrm{diag}(\bA \bA^T)$; as proved for the first time by Elfving and Nikazad \cite{elfving}.
	
	There are some more elaborate versions of Kaczmarz method available,
	for instance using block row partitioning or variable relaxation parameters
	\cite{elfving_2014,gordon}.
	In this paper we propose two new methods where we exploit the sequential version \eqref{eq:ART} of Kaczmarz with a fixed relaxation parameter.
	We believe the techniques may be generalized to the mentioned Kaczmarz schemes, but this is outside the scope of this work.
	In contrast, we will see in Section~\ref{sec:rkm} that it is not obvious to combine the proposed methods with a randomized Kaczmarz scheme.

	From \eqref{eq:Elfving_theorem}, Kaczmarz's method can be seen as an attempt to solve the system
	\begin{equation}\label{eq:augmented_system}
		\bA^T \bL^{-1} \bA \bx = \bA^T \bL^{-1} \bb,
	\end{equation}
	which is always consistent, whether or not \eqref{eq:Axb} is.
	Indeed, supposing $\bx_{k+1}$ approaches $\bx_k$ in the limit of large $k$, we end up with $\bA^T \bL^{-1}(\bb-\bA \bx) = \zero$, which is equivalent to \eqref{eq:augmented_system}.
	We have not come across of the following two lemmas in the literature.
	
	\begin{lemma}\label{lem:nullspace}
		The nullspaces of $\bA^T\bL^{-1}\bA$ and its transpose are identical, and equal to the nullspace of $\bA$.
	\end{lemma}
	\begin{proof}
		While it is straightforward that $\mathcal{N}(\bA) \subseteq \mathcal{N}(\bA^T \bL^{-1} \bA)$, it is also true that $\mathcal{N}(\bA) \supseteq \mathcal{N}(\bA^T \bL^{-1} \bA)$ by the following argument.
		In view of \eqref{eq:Elfving_theorem}, $\bL+\bL^T$ is symmetric positive definite (SPD) for $\omega \in (0,2)$, and therefore so is the congruence transform $\bL^{-1}(\bL+\bL^T)\bL^{-T} = \bL^{-1}+\bL^{-T}$; it follows from Sylvester's law of inertia that all eigenvalues remain positive under such transformations. Let $\bx \neq \zero$ be such that $\bA^T\bL^{-1}\bA\bx = \zero$.
		Left-multiplication by $\bx^T$ and using the fact that the symmetric part of $\bL^{-1}$ is SPD shows that $\bA\bx = \zero$.
		The statement for the transpose follows from an identical reasoning with $\bL^T$ taking the role of $\bL$.
	\end{proof}
	
	\begin{lemma}\label{lem:consistent_equivalent_kaczmarz}
		If \eqref{eq:Axb} is consistent, \eqref{eq:augmented_system} is equivalent to \eqref{eq:Axb}.
	\end{lemma}
	\begin{proof}
		It is obvious that \eqref{eq:Axb} implies \eqref{eq:augmented_system}, the opposite follows the fact that there exists some $\by$ such that $\bA \by = \bb$. From \eqref{eq:augmented_system} we find $\bA^T \bL^{-1} \bA (\bx - \by) = \zero$. Applying Lemma~\ref{lem:nullspace} shows that $\bA\bx = \bA\by = \bb$.
	\end{proof}
	
	Let $\mathcal{V}$ be the row space of $\bA$, which is the orthogonal complement of the nullspace of $\bA$.
	Lemma~\ref{lem:nullspace} asserts that $\mathcal{V}^\perp$ is also the nullspace $\bA^T \bL^{-1} \bA$, so that $\mathcal V$ is the row space, while any linear operator is nonsingular when the kernel is removed. The restriction is related to the pseudoinverse in the following way
	\begin{equation*}
		(\bA^T \bL^{-1} \bA)^{\dagger} = (\bA^T \bL^{-1} \bA\vert_{\mathcal{V}})^{-1} P_\mathcal{V},
	\end{equation*}
	where $P_\mathcal{V}$ denotes the orthogonal projection onto $\mathcal{V}$. The pseudoinverse is defined for any vector, while the domain of the restriction is only $\mathcal{V}$. We will use the restriction for clarity and preciseness.
	Kaczmarz's method is a row-action scheme: all iterates as well their limit are constructed as a linear combination of the rows of $\bA$, i.e., they are in $\mathcal{V}$. This means that Kaczmarz's method provides the minimum-norm solution subject to the constraint that \eqref{eq:augmented_system} holds.
	This solution can be expressed in terms of the inverse of the restricted operator:
	\begin{equation} \label{eq:mp}
		\bx := (\bA^T \bL^{-1} \bA\vert_{\mathcal{V}})^{-1} \bA^T \bL^{-1} \bb.
	\end{equation}
	Since $\boldsymbol{\Pi} := \bA (\bA^T \bL^{-1} \bA\vert_{\mathcal{V}})^{-1} \bA^T \bL^{-1}$ is an oblique projection onto the span of $\bA$, we observe that
	$\bA\bx = \boldsymbol{\Pi}\bb$, the image of $\bb$ under this oblique projection.
	In the case of a consistent system \eqref{eq:Axb}, $\bb$ is in the span of $\bA$ and $\boldsymbol{\Pi} \bb= \bb$, reaffirming Lemma~\ref{lem:consistent_equivalent_kaczmarz} from a different viewpoint.
	We note that $\bx$ can also be seen as the orthogonal projection onto the row space of $\bA$ of an arbitrary $\by$ satisfying $\bA\by = \boldsymbol{\Pi}\bb$.
	
	Using \eqref{eq:Elfving_theorem} we can interpret Kaczmarz's method as a fixed-point method, since we have
	\begin{equation}\label{eq:fixed_point_form}
		\bx_{k+1} = \bG \, \bx_k + \bA^T \bL^{-1} \bb, \qquad
		\bG := \bI - \bA^T \bL^{-1} \bA ,
	\end{equation}
	where $\bI$ is the identity matrix. Note that the iteration matrix $\bG$ is independent of the row scaling of $\bA$; if we premultiply $\bA$ by any diagonal matrix, $\bG$ is unaltered.
	Moreover, we point out that $\bG$ is the identity on $\mathcal{V}^\perp$, the nullspace of $\bA$. For any initial guess $\bx_0 \in \mathcal{V}$, we can without loss of generality restrict $\bG$ to $\mathcal{V}$. In particular, our initial guess is always $\bx_0 = \zero$, so that we can always restrict $\bG$ to $\mathcal{V}$ with impunity.
	We emphasize that $\bG$ depends on the relaxation parameter $\omega$ via~$\bL$; cf.~\eqref{eq:Elfving_theorem}.
	Using \eqref{eq:augmented_system}, we find
	\begin{equation}\label{eq:fixed_point_down-sweep}
		\bx_{k+1} - \bx = \bG \, \bx_k - \bx + \bA^T \bL^{-1}
		\bA \, \bx = \bG \, (\bx_k - \bx),
	\end{equation}
	where $\bx$ is given by \eqref{eq:mp}.
	For fixed-point methods, we need the following well-established result, which we adapt from \cite[Thm.~4.1]{saad}.
	
	\begin{lemma}\label{lem:saad}
		Kaczmarz's method with $\bx_0 \in \mathcal{V}$ (so particularly for our choice $\bx_0 = \zero$) converges to \eqref{eq:mp} if and only if the spectral radius of the restricted iteration matrix is strictly smaller than 1, i.e., $\rho(\bG\vert_\mathcal{V})<1$.
	\end{lemma}
	
	Of course, the convergence of Kaczmarz's method has been studied since its introduction; we recall the following well-known result from \cite{elfving}.
	
	\begin{lemma}\label{lem:consisent_kaczmarz}
		Kaczmarz's method is convergent for linear system \eqref{eq:Axb} for any $\omega \in (0,2)$. If $\bb$ lies in the span of $\bA$, then the method converges to a solution of $\eqref{eq:Axb}$. If in addition $\bx_0 \in \mathcal{V}$, then the Kaczmarz converges to the minimal norm solution to \eqref{eq:Axb}.
	\end{lemma}
	
	The two lemmas combined show that Kaczmarz's method always converges to a solution of \eqref{eq:augmented_system}, and if $\bx_0 \in \mathcal{V}$, then it provides the minimum-norm solution. From Lemma~\ref{lem:consistent_equivalent_kaczmarz}, we see that if \eqref{eq:Axb} is consistent, \eqref{eq:augmented_system} is equivalent to \eqref{eq:Axb} and Kaczmarz provides a solution to \eqref{eq:Axb}.
	
	Additionally, the two lemmas also combine to inform us on the eigenvalues of $\bG$. The following result, which follows from the fact that Kaczmarz's method converges, is in the line of \cite{tanabe}, but stated in terms of the restricted operator $\bG$.
	
	\begin{theorem}\label{thm:eigenvalues_G}
		The restricted iteration matrix of Kaczmarz's method $\bG\vert_\mathcal{V} = (\bI - \bA^T \bL^{-1} \bA)\vert_\mathcal{V}$  satisfies $\rho(\bG\vert_\mathcal{V})<1$ for any $0<\omega<2$ and any $\bA$.
	\end{theorem}
	
	Aside from any consistency of \eqref{eq:Axb}, we also have to deal with noise, which is a slightly more subtle concept.
	Noise in the right-hand side $\bb$ will generally have components both inside and outside of the span of $\bA$. The reconstruction $\bx$ may be greatly affected by the noise, but, of course, only by the noise component in the span of $\bA$. We now introduce several concepts to analyze this situation. Denote the noise-free image by $\bx_\mathrm{nf}$, and define $\bb\ex := \bA \bx_\mathrm{nf}$. If we would execute the Kaczmarz process until convergence using $\bb\ex$ as right-hand side, we would obtain a solution $\bx\ex$ satisfying $\bx\ex = (\bA^T\bL^{-1}\bA\vert_{\mathcal{V}})^{-1} \bA^T \bL^{-1} \bb\ex$. We refer to $\bx\ex$ as the noise-free solution. As $\bb\ex$ is in the span of $\bA$, we conclude from Lemma~\ref{lem:consistent_equivalent_kaczmarz} that $\bA \bx\ex = \bb\ex$, and that $\bx\ex$ is the orthogonal projection of $\bx_\mathrm{nf}$ onto the row space of $\bA$. Given a measured data vector $\bb$, the noise vector is defined by $\delta \bb := \bb - \bb\ex$. In our numerical experiments in Section~\ref{sec:numerical_experiments}, we simulate $\delta \bb$ as a Gau{\ss}ian random vector with zero mean and standard deviation $\std$. Unlike the stopping rules we review in Section~\ref{sec:stoprules}, which require normally distributed noise to work, the methods we propose do not need any assumptions on the statistical nature of the noise.
	
	We also make a corresponding splitting in the reconstruction iterates by introducing $\bbx_k$ as the iterates of Kaczmarz's method with right-hand side $\bb\ex$. The iteration error is then defined as $\bbx_k - \bx\ex$, while the noise error is defined as $\bx_k - \bbx_k$. Obviously, the error with respect to the noise-free solution is the sum of the iteration error and the noise error, i.e.,
	\begin{equation}
		\be_k = \bx_k - \bx\ex = \underbrace{\bx_k - \bbx_k}_{\text{\normalfont noise error}} + \underbrace{\bbx_k - \bx\ex}_{\text{\normalfont iteration error}}.
	\end{equation}
	As we will see in Section~\ref{sec:errorgaugeandsemi}, the iteration error diminishes while the noise error grows as the iteration progresses.

	\subsection{Up-sweep Kaczmarz}
	
	The up-sweep version of Kaczmarz's method relates in an appealing way to the down-sweep method. The following result is implicit in Elfving and Nikazad \cite{elfving}; we merely provide an explicit demonstration.
	
	\begin{proposition} The up-sweep iteration matrix $\widetilde{\bG}$ is related to the down-sweep iteration matrix $\bG$ \eqref{eq:Elfving_theorem} by transposition, i.e.,
		\begin{equation}
			\widetilde{\bG} = \bG^T.
		\end{equation}
	\end{proposition}
	\begin{proof}
		Let $\bJ$ denote the reverse identity, i.e., the permutation matrix that reverses the ordering. We investigate what happens when we apply the cyclical down-sweep Kaczmarz's method to the system $\bJ \bA \bx = \bJ \bb$. Note first that $\bJ^T = \bJ = \bJ^{-1}$, so that $\bJ \bA \bA^T \bJ^T = \bJ \bA \bA^T \bJ $,
		where pre- and post-multiplying $\bA$ by $\bJ$ results in the flipping over of both the columns and the rows respectively. From this, we see that
		\begin{equation*}
			\widetilde{\bD} =
			\mathrm{diag} (\bJ \bA \bA^T \bJ) = \bJ \, \mathrm{diag} (\bA \bA^T) \, \bJ, \tag{a}
		\end{equation*}
		since flipping over both columns and rows of a diagonal matrix yields a diagonal matrix with the entries reversed. Here, $\widetilde{\bD}$ is the up-sweep analogue of $\bD$. Next, it may be checked that the strictly lower triangular part of some matrix $\bJ \bM \bJ$ is the strictly {\em upper} triangular part of $\bM$, flipped over both columns and rows. Hence, the relation is exactly $\code{slt}(\bJ \bM \bJ) = \bJ \, \code{sut}(\bM) \, \bJ$, where \code{sut} takes the strictly upper triangular part. Applying this to the symmetric matrix $\bA \bA^T$, we obtain
		\begin{equation*}\label{eq:flipped_AAT}
			\code{slt} (\bJ \bA \bA^T \bJ) = \bJ \, \code{sut} (\bA \bA^T)\bJ = \bJ \, \code{slt} (\bA \bA^T)^T \bJ. \tag{b}
		\end{equation*}
		Putting (a) and (b) together, we find that
		\begin{equation*}
			\widetilde{\bL} = \bJ \big( \code{slt} (\bA \bA^T)^T  + \tfrac{1}{\omega}  \bD \big) \bJ = \bJ \bL^T \bJ,
		\end{equation*}
		where $\widetilde{\bL}$ is the up-sweep analogue of $\bL$.
		Thus, when we inspect the up-sweep iteration matrix $\widetilde{\bG}$, we find
		\begin{equation*}
			\widetilde{\bG} := \bI - (\bJ \bA)^T \big(\bJ \bL^{-T} \bJ \big) (\bJ \bA) = \bI - \bA^T \bL^{-T} \bA.
		\end{equation*}
		We identify this last matrix as $\bG^T$, which completes the proof.
	\end{proof}
	
	\begin{keypoint}The iteration matrix of the up-sweep Kaczmarz method is the transpose of the down-sweep iteration matrix. Consequently, they have \emph{the same eigenvalues}.
	\end{keypoint}
	
	\subsection{Statistical stopping rules} \label{sec:stoprules}
	To use the semi-convergence of Kaczmarz's method for noisy
	data we need a stopping rule for terminating the
	iterations near the point of semi-convergence.
	Ideally we prefer a stopping rule based on an estimate of the reconstruction error
	$\| \bx_k - \bx\ex \|$.
	Several such rules have been proposed for regularization methods
	that can be implicitly expressed as a filtered SVD expansion
	\cite[Sec.~7.3]{Hansen1998}; they do not apply to Kaczmarz's method, as Kaczmarz cannot be represented in this way.
	The stopping rule presented in this work does not have this limitation.
	
	Instead of estimating the reconstruction error, several statistical stopping rules based on the prediction error have been proposed.
	These rules, which indeed apply to Kaczmarz's method, seek to minimize the norm of the prediction error
	for the $k$th iteration, defined as
	\begin{equation}
		\prederr_k = \bb\ex - \bA \, \bx_k \ .
	\end{equation}
	However, since this is unavailable, the methods work instead with the norm of the residual vector $\bb - \bA\,\bx_k$.
	One way to do so involves the trace of the \emph{influence matrix} $\bA \bA^\#_k$, where $\bx_k = \bA^\#_k \bb$ and $\bA^\#_k$ is the action of Kaczmarz's method.
	The trace of this matrix features in many stopping rules. For iterative regularization methods the trace is often estimated by means of a Monte Carlo approach as proposed in \cite{Girard89} and \cite{SantosPierro03}. This is done using a normally distributed random vector $\boldsymbol{\xi}_0$, which is used as an initial guess for Kaczmarz's method applied to the system $\bA \boldsymbol{\xi} = \zero$. The inner product of $\boldsymbol{\xi}_0$ with the iterate $\boldsymbol{\xi}_k$ provides the trace estimate. Employing the quadratic form identity \cite[Thm.~5.2a]{rencher}, it can be shown that
	\begin{equation}
		\mathbb{E} (\boldsymbol{\xi}_0^T \boldsymbol{\xi}_k) = m -\tr(\bA \bA_k^\#).
	\end{equation}
	A Monte Carlo method typically draws many samples to estimate a quantity. However, each sample would require another instance of Kaczmarz's method, quickly becoming very expensive. In practice often only a single sample is drawn.

	We can now summarize the three statistical stopping rules that we compare with in this work.
	In the \emph{unbiased predictive risk estimation} (UPRE) method we find the $k$ that minimizes the expected prediction estimation error norm $\mathbb{E}( \| \prederr_k \|^2 )$.  This is done by minimizing the quantity
	\begin{equation}
		\| \bb - \bA\,\bx_k \|^2 + 2\,\std^2 \, \tr ( \bA \bA^\#_k ) - \std^2 \, m \ .
	\end{equation}
	The \emph{generalized cross validation} (GCV) method also seeks to minimize the expected prediction error, and it does so without the need for the noise's standard deviation~$\std$. Here we find the $k$ that minimizes
	\begin{equation}
		\frac{ \| \bb - \bA\,\bx_k \|^2 }{ \bigl(m - \tr ( \bA \bA^\#_k ) \bigr)^2 } .
	\end{equation} 
	The \emph{compensated discrepancy principle} (CDP) was defined by Turchin \cite{turchin} for Tikh\-o\-nov regularization. The underlying idea is to determine the largest iteration number for which we cannot reject $\bx_k$ -- computed from the noisy data -- as a possible solution to the noise-free system, cf.~\cite[p.~93]{turchin}.
	Here we stop at the first iteration $k$ for which
	\begin{equation}
		\| \bb - \bA\,\bx_k \|^2 \leq \std^2\,\bigl( m-\tr ( \bA \bA^\#_k ) \bigr) \ .
	\end{equation}
	A derivation of UPRE is given in \cite[Sec.~7.1]{vogel} while summaries of
	GCV and CDP can be found in \cite[Secs.~7.2 and 7.4]{Hansen1998}. The methods we reviewed here require the assumption that $\delta\bb$ is white Gau{\ss}ian noise, while the methods we propose do not require any assumptions on the statistical nature of the noise.
	
	\section{The Error Gauge and the \TM\label{sec:error_gauge_TA}}
	\label{sec:gauge}
	Given two numerical methods designed to solve the same problem, a general approach to error estimation is to take the difference of the two numerical solutions $\bx_k$ and $\bup_k$. One way to reason about this is by adding and subtracting the noise-free solution $\bx\ex$.
	\begin{equation}\label{eq:error_gauge}
		\begin{aligned}
			\| \bx_k - \bup_k \|^2 &= \| \bx_k - \bx^\star + \bx^\star - \bup_k \|^2 \\
			&= \| \be_k - \widetilde{\be}_k \|^2 \\
			&= \| \be_k \|^2 + \| \widetilde{\be}_k \|^2 - 2 \cos(\varphi_k) \,
			\| \be_k \| \, \| \widetilde{\be}_k \|,
		\end{aligned}
	\end{equation}
	where $\be_k$ and $\widetilde{\be}_k$ are their respective errors,
	and $\varphi_k$ is the angle between the two errors $\be_k$ and $\widetilde{\be}_k$.
	We will exploit the expression involving the cosine in Section~\ref{sec:errorgaugeandsemi}.

	\subsection{Analysis of an error gauge: the noise-free consistent case}
	\label{sec:analysis}
	Here, we propose to employ \eqref{eq:error_gauge} as an \emph{error gauge}, which estimates the accuracy of the iterates.
	When the approximations are different, but converge to the same solution at the same rate,
	the difference between them will vanish at the same rate.
	Therefore this gives us a simple error gauge.
	
	To obtain two different iterates $\bx_k$ and $\bup_k$ of Kaczmarz's
	method to be used in \eqref{eq:error_gauge}, we use down-sweeps and up-sweeps, respectively.
	
	\begin{keypoint}
		For Kaczmarz's method, the iteration matrix $\bG$ is generally not symmetric, even for symmetric $\bA$.
	\end{keypoint}
	
	Before stating the results, we recall that
	the condition number of a simple eigenvalue $\lambda$ is given by
	\begin{equation}
		\label{eq:kappa}
		\kappa(\lambda) = |\widetilde\bv^H \bv|^{-1},
	\end{equation}
	where $\widetilde\bv$ and $\bv$ are the normalized left and right eigenvectors, respectively \cite[Chap.~1, Sec.~3.2]{stewart}, and
	$\boldsymbol{\cdot}^H$ denotes the conjugate transpose.
	A simple eigenvalue $\lambda$ is called normal if $\kappa(\lambda) = 1$, which is the case if and only if $\bv$ and $\widetilde\bv$ coincide.
	In the following proposition we make two assumptions. First, we assume that $\bG$ is diagonalizable, which means that it has an eigenvalue decomposition;
	second, that $\lambda_1$ is a simple and nonnormal eigenvalue.
	Note that these two assumptions are generic properties, holding almost always.
	\footnote{A necessary condition for matrices to be nondiagonalizable is to have zero discriminant, defined by $\mathrm{dis}(\bG) = \prod_{i < j} (\lambda_i - \lambda_j)^2$ \cite[2.4.P21]{horn_johnson}.
		The set of matrices with $\mathrm{dis}(\bG) = 0$ constitutes a set of zero measure in the $n^2$-dimensional space, meaning that the first assumption indeed is a generic property.
		For the second assumption, given $\bG$ with its eigenpair $(\lambda, \bv)$ consider the Schur decomposition $\smtxa{cc}{\lambda_1 & \by^* \\ \zero & \bM}$, corresponding to any basis $[\bv \ \, \bV_{\perp}]$, where $\bV_{\perp} \in \mathbb C^{n \times (n-1)}$, $\bV_{\perp}^*\bv = \zero$, $\bM \in \mathbb C^{(n-1) \times (n-1)}$, and $\by \in \mathbb C^{n-1}$. Then $\lambda_1$ is a normal eigenvalue (so with $\widetilde \bv = \bv$) if and only if $\by = \zero$. This means that the property of being a simple and nonnormal eigenvalue is a generic property, since the constraint $\by = \zero$ also results in a set of zero measure.
	}.
	
	We now present two propositions. For the first, we assume that the eigenvalues of $\bG$ can be labeled according to their modulus as
	\begin{equation}
		|\lambda_n|  \leq \dots \leq |\lambda_2| < |\lambda_1| <1.     
	\end{equation}
	This implies that $\lambda_1$ is real. This assumption holds for many of the experiments that we carry out in Section~\ref{sec:numerical_experiments}.
	
	\begin{proposition}\label{prop:good_estimate}
		Suppose the iteration matrix $\bG \in \mathbb{R}^{n\times n}$ has an eigendecomposition and the largest eigenvalue in modulus $\lambda_1$ is isolated, with $|\lambda_1| > |\lambda_i|$ for all $i > 1$.
		We furthermore assume that $\lambda_1$ is simple and nonnormal and we denote its right and left eigenvector by $\bv_1$ and $\widetilde{\bv}_1$, respectively.
		Suppose that $\be_0$, the initial error for the down-sweep iterates, has a nonzero component $\gamma_1$ in the direction of $\bv_1$, and that the same holds for $\widetilde{\gamma_1}$, the component of $\widetilde{\be}_0$ in the direction of $\widetilde{\bv}_1$.
		
		Then, for consistent systems, the error gauge
		$\| \bx_k - \bup_k \|$ is asymptotically
		proportional to the true error norms $\| \be_k\|$ and $\| \widetilde{\be}_k \|$, i.e.,
		\begin{equation}\label{eq:asymptotic_proportionality}
			\begin{aligned}
				\frac{\| \bx_k - \bup_k \|}{\| \be_k \|} & = 
				\frac{\| \gamma_1 \bv_1 - \widetilde \gamma_1 \widetilde{\bv}_1 \|}{ |\gamma_1|  } \, 
				+ \mathcal{O}\left( \left| \frac{\lambda_2}{\lambda_1} \right|^{k} \right), \\
				\frac{\| \bx_k - \bup_k \|}{\| \widetilde{\be}_k \|} & = 
				\frac{\| \gamma_1 \bv_1 - \widetilde \gamma_1 \widetilde{\bv}_1 \|}{ |\widetilde \gamma_1| } \,  
				+ \mathcal{O}\left( \left| \frac{\lambda_2}{\lambda_1} \right|^{k}      \right)
			\end{aligned}
		\end{equation}
		for large $k$.
		Moreover, the right-hand sides are nonzero.
	\end{proposition}
	\begin{proof}
		The system is consistent, so that \eqref{eq:fixed_point_down-sweep} holds, i.e.,
		\begin{equation*}
			\be_{k+1} = \bG\be_k \qquad \hbox{and} \qquad
			\widetilde{\be}_{k+1} = \bG^T \widetilde{\be}_k .
		\end{equation*}
		The eigendecomposition allows us to write
		\begin{equation*}
			\be_k   = \lambda_1^k \, \gamma_1\,  \bv_1 + \sum_{j=2}^n \lambda_j^k \, \gamma_j \, \bv_j = \lambda_1^k \, \Big( \gamma_1 \bv_1 + \sum_{j=2}^n \left(\frac{\lambda_j}{\lambda_1}\right)^k \gamma_j \bv_j \Big).
		\end{equation*}
		The error of the up-sweeps admits an analogous expression. From the assumption that $\lambda_1$ is isolated, it follows that
		\begin{equation*}
			\| \be_k \| = | \lambda_1|^k \, \big( |\gamma_1 | + \mathcal{O} \big( \big| \tfrac{\lambda_2}{\lambda_1} \big|^k \big) \big).
		\end{equation*}
		However, we also have 
		\begin{equation*}
			\| \bx_k - \bup_k \| = \|  \be_k - \widetilde{\be}_k \|
			= | \lambda_1 |^{k} \, \left( \| \gamma_1 \bv_1 - \widetilde{\gamma_1} \widetilde{\bv}_1 \| + \mathcal{O}\big( \big| \tfrac{\lambda_2}{\lambda_1} \big|^{k} \big) \right).
		\end{equation*}
		The assumption that $\lambda_1$ is nonnormal is equivalent to
		$\bv_1$ and $\widetilde{\bv}_1$ being linearly independent.
		Consequently, we have $\| \gamma_1 \bv_1 - \widetilde{\gamma_1} \widetilde{\bv}_1 \|>0$,
		since the only way to produce a zero constant is to have $\gamma_1 = \widetilde{\gamma_1} = 0$, which by assumption does not occur.
		Therefore we can conclude \eqref{eq:asymptotic_proportionality}.
	\end{proof}
	
	Next, we consider the situation that $\lambda_1$ is non-real and therefore $\lambda_2 = \overline{\lambda_1}$, the complex conjugate of $\lambda_1$, so that $|\lambda_2| = |\lambda_1|$.
	This case may also arise in our numerical experiments for some choices of the problem parameters. Since $\lambda_1$ is usually (very) close to 1 and inside the complex unit circle, the imaginary part is very modest.
	We make the generic assumption that $|\lambda_1| > |\lambda_3|$.
	
	\begin{proposition}\label{prop:good_estimate2}
		Given the context of Proposition~\ref{prop:good_estimate}, but now with $\lambda_2 = \overline{\lambda_1}$, and $|\lambda_1| > |\lambda_3|$, we have
		\begin{equation*}
			\begin{aligned}
				\frac{\| \bx_k - \bup_k \|}{\| \be_k \|} & = 
				\frac{\| \gamma_1 \bv_1 + \overline \gamma_1 \overline \bv_1
					- \widetilde \gamma_1 \widetilde \bv_1 - \overline{\widetilde \gamma_1} \, \overline {\widetilde \bv_1} \|}{ \|\gamma_1 \bv_1 + \overline \gamma_1 \overline \bv_1\| } \,  + \mathcal{O}\big( \big| \tfrac{\lambda_3}{\lambda_1} \big|^{k} \big), \\
				\frac{\| \bx_k - \bup_k \|}{\| \widetilde{\be}_k \|} & = 
				\frac{\| \gamma_1 \bv_1 + \overline \gamma_1 \overline \bv_1
					- \widetilde \gamma_1 \widetilde \bv_1 - \overline{\widetilde \gamma_1} \, \overline {\widetilde \bv_1} \|}{ \|\widetilde \gamma_1 \widetilde \bv_1 + \overline{\widetilde \gamma_1} \, \overline{\widetilde \bv_1}\| } \, + \mathcal{O}\big( \big| \tfrac{\lambda_3}{\lambda_1} \big|^{k} \big)
			\end{aligned}
		\end{equation*}
	\end{proposition}
	\begin{proof}
		This follows easily from the fact that the asymptotically dominant term of $\be_k$ is 
		$\lambda_1^k \gamma_1 \bv_1 + \overline {\lambda}_1^k \, \overline \gamma_1 \overline{\bv_1}$, with a similar expression for $\widetilde e_k$.
	\end{proof}
	
	\begin{keypoint}
		The error gauge $\| \bx_k - \bup_k \|$
		depends crucially on the fact that $\bG$ is not symmetric.
	\end{keypoint}
	
	Put in colloquial terms, Propositions~\ref{prop:good_estimate} and \ref{prop:good_estimate2} assert that the error gauge
	$\| \bx_k - \bup_k \|$ is a good estimate of the iteration error of a consistent system.
	The intuitive argument is simple: both methods converge to the same solution at the same rate, but along different paths, so that the difference vanishes at the same rate as the errors.
	
	\begin{figure}
		\centering
		\includegraphics[width=\textwidth]{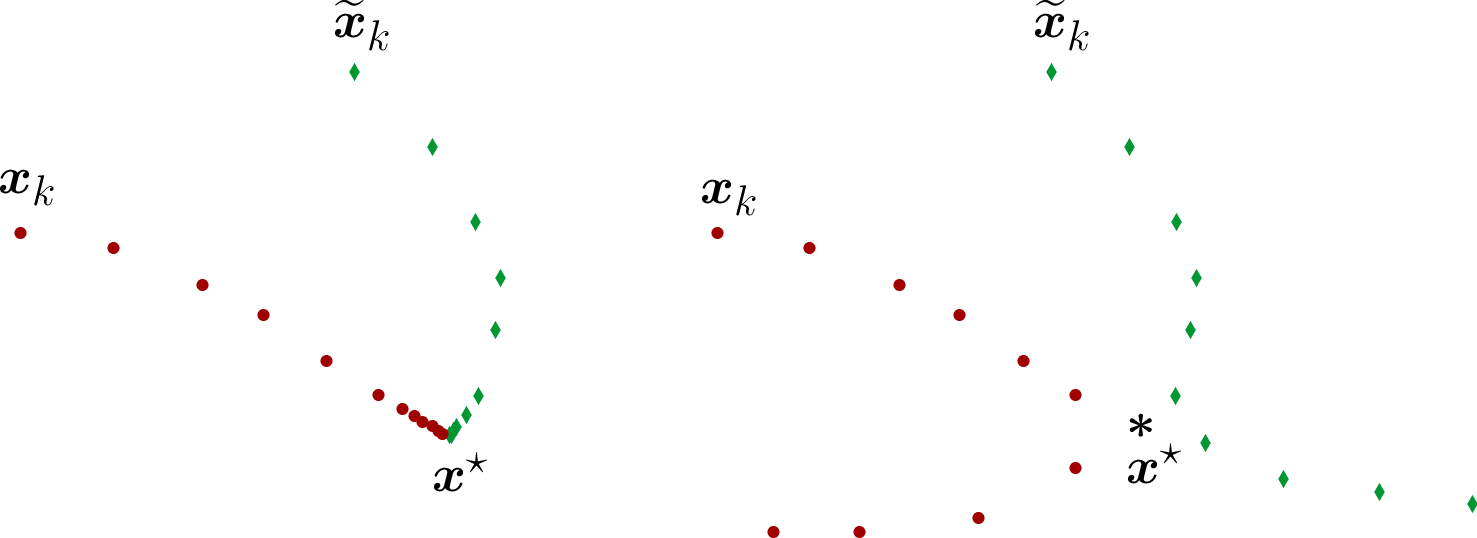}
		\caption{Sketch of semi-convergence and the error gauge. Left: consistent case, both sequences converge to the same point.
			Right: noisy case, the noise error causes a divergence away from the noise-free solution $\bx^\star$.}
		\label{fig:error_estimate_sketch}
	\end{figure}

	\subsection{Analysis of the noisy case}
	
	Having established the error gauge $\| \bx_k - \bup_k \|$ we now turn to an analysis of the behavior of Kaczmarz's method for noisy data.
	
	\subsubsection{Semi-convergence}
	
	Kaczmarz's method applied to noisy CT problems exhibits semi-convergence: the error first decreases and the iterates approach the noise-free solution, after which the noise component starts to dominate and the iterates move away from the noise-free solution. The minimal error is known as the semi-convergence point. The semi-convergence behavior is sketched in Figure~\ref{fig:error_estimate_sketch}.
	The nature of the first part of this subsection may be considered folklore of Kaczmarz type methods; 
	we try to present a thorough analysis here.
	
	Recall that the data vector is split according to $\bb = \bb\ex + \delta \bb$ with a noise-free consistent part $\bb\ex$ and a noise part $\delta \bb$. We also recall that analogously the error is split into an iteration error $\bbx_k - \bx\ex$ and a noise error $\bx_k - \bbx_k$, where $\bbx_k$ are the iterates of the noise-free consistent system using $\bb\ex$ as the right-hand side.
	When we use an empirical data vector $\bb$ instead of a noise-free consistent data vector $\bb\ex$ in
	Kaczmarz's down-sweep method, we obtain from \eqref{eq:fixed_point_form}
	\begin{equation}
		\bx_k  = (\bI+\bG + \cdots + \bG^{k-1}) \, \bA^T \bL^{-1}(\bb\ex + \delta \bb) \nonumber, \label{eq:polynomial_form}
	\end{equation}
	Since $\bA^T \bL^{-1} (\bb\ex + \delta \bb) \in \mathcal{V}$, we can restrict the geometric sum to $\mathcal{V}$ and get
	\begin{equation}
		\bx_k = (\bI - \bG^k) (\bA^T\bL^{-1}\bA\vert_{\mathcal{V}})^{-1} \, \bA^T \bL^{-1}(\bb\ex + \delta \bb).
	\end{equation}
	We now employ the identity $\bb\ex = \bA \bx\ex$ and use the fact that $\bx\ex$ is in the row space of $\bA$, yielding
	\begin{equation}\label{eq:exact_iterate}
		\bx_k = (\bI - \bG^k)\,\bx\ex +  (\bI - \bG^k) (\bA^T\bL^{-1}\bA\vert_{\mathcal{V}})^{-1} \, \bA^T \bL^{-1} \delta \bb.
	\end{equation}
	We thus arrive at an exact expression for the error
	\begin{equation}\label{eq:exact_error_behavior}
		\bx_k - \bx\ex = -\bG^k \bx\ex + (\bI - \bG^k) (\bA^T\bL^{-1}\bA\vert_{\mathcal{V}})^{-1} \bA^T \bL^{-1} \delta \bb .
	\end{equation}
	Hence, the splitting in the empirical data vector leads to the splitting of the iteration error and the noise error, i.e.,
	\begin{subequations}
		\begin{align}
			\bbx_k - \bx\ex &= - \bG^k \bx\ex, \label{eq:exact_iteration_error} \\
			\bx_k - \bbx_k &= (\bI - \bG^k) \bf, \label{eq:exact_noise_error}
		\end{align}
	\end{subequations}
	where we have defined (cf.~\eqref{eq:mp})
	\begin{equation}\label{eq:inverted_noise}
		\bf = (\bA^T\bL^{-1}\bA\vert_{\mathcal{V}})^{-1} \bA^T \bL^{-1} \delta \bb.
	\end{equation}
	This is the limiting vector of the noise error iterates for $k \to \infty$, and it therefore may be called the ``inverted noise.'' The expression \eqref{eq:inverted_noise} should be compared with \eqref{eq:mp}; it represents exactly the part that one would like to suppress in a regularized solution.
	
	We will now show that not only $\rho(\bG\vert_\mathcal{V})<1$, but also the stronger result that $\|\bG\vert_{\mathcal{V}}\| <1$. (It is stronger since any consistent norm is an upper bound for the spectral radius.)
	As for Theorem~\ref{thm:eigenvalues_G}, we derive this property from a known fact of Kaczmarz type methods.
	
	\begin{lemma}\label{lem:Gnorm_bound}The iteration matrix $\bG$ satisfies $\| \bG\vert_{\mathcal{V}} \| <1$.
	\end{lemma}
	\begin{proof}
		Both $\bG$ and its transpose maps $\mathcal{V}$ onto itself.
		The two-norm of $\bG\vert_{\mathcal{V}}$ is given by $\sqrt{\rho(\bG^T \bG\vert_{\mathcal{V}})}$, but $\bG^T \bG$ is the iteration matrix of the method known as the symmetric Kaczmarz method from Elfving and Nikazad \cite[Prop.~11]{elfving}. This method is convergent and applying Lemma~\ref{lem:saad} concludes the proof.
	\end{proof}
	
	Kaczmarz's method is convergent, and Lemma~\ref{lem:Gnorm_bound} shows that it is even monotonically convergent, as $\|\be_{k+1}\| = \|\bG\be_k\| \leq \|\bG\vert_\mathcal{V}\| \, \|\be_k\| < \|\be_k\|$. This also implies that the noise error forms a monotonic sequence with respect to its limiting vector $\bf$.
	
	\begin{proposition}\label{prop:noise_error_norm} The noise error satisfies
		\begin{equation}\label{eq:noise_error_convergence}
			\bx_{k+1} - \bbx_{k+1} - \bf = \bG \, (\bx_{k} - \bbx_{k} - \bf ),
		\end{equation}
		with $\bf$ from \eqref{eq:inverted_noise}.  Moreover, $\|\bx_{k+1} - \bbx_{k+1} - \bf\|$ constitutes a monotonically decreasing sequence, so that
		\begin{equation}
			\lim_{k \to \infty} \bx_{k} - \bbx_{k}  = \bf.
		\end{equation}
	\end{proposition}
	
	\begin{proof}
		From \eqref{eq:fixed_point_form}, the noise error is given by
		\begin{equation*}
			\bx_{k+1} - \bbx_{k+1} = \bG \, (\bx_{k} - \bbx_{k}) + \bA^T \bL^{-1} \delta \bb.
		\end{equation*}
		Since $\bG = \bI - \bA^T \bL^{-1} \bA$, we can write
		\begin{equation*}
			\bx_{k+1} - \bbx_{k+1} = \bG \, (\bx_{k} - \bbx_{k}) + (\bI - \bG)\,\bf,
		\end{equation*}
		which resolves into \eqref{eq:noise_error_convergence}. After taking norms, applying Lemma~\ref{lem:Gnorm_bound} completes the proof.
	\end{proof}
	
	The following proposition is, to the best of our knowledge, new. It shows that the noise error must increase, which is a fundamental insight into the semi-convergence behavior of Kaczmarz's method.
	
	\begin{proposition}\label{prop:noise_error_monotone_bound} The norm of the noise error is bounded from below by
		\begin{equation}\label{eq:noise_error_lower_bound}
			\| \bx_k - \bbx_k \| \geq  \big(1- \| \bG\vert_{\mathcal{V}} \|^k \big) \| \bf \| ,
		\end{equation}
		where $\bf$ is defined in \eqref{eq:inverted_noise}. Moreover, if $\bA^T \bL^{-1} \delta \bb \neq \boldsymbol{0}$, then the lower bound forms a monotonically increasing sequence.
	\end{proposition}
	
	\begin{proof}
		We take the norm and apply the reverse triangle inequality
		\begin{equation*}
			\| \bx_k - \bbx_k \| \geq \big| \| \bf \| - \| \bG^k \bf \| \big|.
		\end{equation*}
		We can now use the fact that $\| \bG^k \bf \| \leq \| \bG\vert_{\mathcal{V}} \|^k \| \bf \|$, with Lemma~\ref{lem:Gnorm_bound} asserting that $\| \bG\vert_{\mathcal{V}} \|<1$, to find \eqref{eq:noise_error_lower_bound}. Finally, the right-hand side of \eqref{eq:noise_error_lower_bound} is a monotonically increasing sequence if $\bf \neq \boldsymbol{0}$, which is equivalent to $\bA^T \bL^{-1} \delta \bb \neq \boldsymbol{0}$.
	\end{proof}

	Propositions~\ref{prop:noise_error_norm} and \ref{prop:noise_error_monotone_bound} show roughly where semi-convergence comes from: the iteration error decreases while the noise error increases. When that happens in just the right way, the error goes through a minimum. In the following, we employ a more phenomenological approach.

	\subsubsection{The error gauge and semi-convergence}
	\label{sec:errorgaugeandsemi}
	
	We can use \eqref{eq:polynomial_form} to define pseudo-iterates for any real number $\tau \geq 0$, using the spectral decomposition (or Jordan normal form) to define $\bG^\tau$. Let $f$ be the error of the pseudo-up-sweeps and $g$ be the error of the pseudo-down-sweeps, then the errors of the iterates are discrete samples of the continuous functions $f$ and $g$, i.e.,
	\begin{equation}
		f(k) = \| \be_k \|, \quad g(k) = \| \widetilde{\be}_k \|.
	\end{equation}
	The discrete error behavior of Kaczmarz's method is typically very benign, with the error being convex up to an inflection point, after which it is concave and reaches its asymptote from below. We assume that $f$ and $g$ exhibit these features in a continuous way.
	
	\begin{assumptions}
		We assume that $f$ and $g$ behave in the following way. There exists a $T>0$, such that $f:[0,T] \to \mathbb{R}$ and $g:[0,T] \to \mathbb{R}$ satisfy:
		\begin{enumerate}
			\item $f \geq f_\mathrm{min} >0$ and $g \geq g_\mathrm{min} >0$.
			\item $f^{\prime \prime} \geq c >0$ and $g^{\prime \prime} \geq c >0$.
			\item Both $f$ and $g$ attain their global minima in $(0,T)$.
		\end{enumerate}
	\end{assumptions}
	
	The minimizers of $f$ and $g$ give us directly the minimizers of $\| \be_k \|$ and $\| \widetilde{\be}_k \|$ by rounding to the nearest integer.
	
	The following is a well-known result from convex analysis (see, e.g., \cite[Sec.~3.2.1]{boyd}) We provide the proof, which is elementary, for completeness.
	
	\begin{lemma} \label{lem:minimum_location_sum_convex}
		Let $\phi:[a,b] \to \mathbb{R}$ and $\psi:[a,b] \to \mathbb{R}$ be strictly convex with unique minimizers in $(a,b)$. Let $\Phi$ be a nonnegative, nonzero weighted sum of $\phi$ and $\psi$, i.e., $\Phi= \alpha \phi+ \beta \psi$, with $\alpha \geq 0$, $\beta \geq 0$ while $\alpha + \beta >0$. Then, $\Phi$ has a unique minimum which lies in between the minimizers of $\phi$ and $\psi$.
	\end{lemma}
	
	\begin{proof}
		Let us say the minimum of $\phi$ occurs at $t_\phi$ and the minimum of $\psi$ occurs at $t_\psi$. Assume without loss of generality that $t_\phi \leq t_\psi$. Then, for $t< t_\phi$ we have $\phi^\prime(t) < 0$ and $ \psi^\prime (t) <0$, so that $\Phi^\prime(t) <0$. Likewise, for $t>t_\psi$, we have $\phi^\prime(t)>0$ and $\psi^\prime(t) > 0$ so that $\phi^\prime(t)>0$. Hence, $\Phi$ must have at least one minimum in $[t_\phi,t_\psi]$ by the intermediate value theorem. However, $\Phi$ is also strictly convex so that this minimum is unique.
	\end{proof}
	
	We now employ some general results from perturbation analysis of optimization problems to show that the minimizer of the error gauge will be close to the minimizers of the two continuous errors $f$ and $g$. The proof of the following result, which provides an upper bound, largely follows the line that of \cite[Prop.~4.32]{shapiro}.
	
	\begin{proposition}
		\label{prop:shift_local_bound}
		Let $S:[a,b] \to \mathbb{R}$ be a strongly convex function with convexity modulus $c$ and a perturbation $p:[a,b] \to \mathbb{R}$ be continuously differentiable. Let $t^\star$ be the minimizer of $S$ and $\tilde{t}$ be the global minimizer of $h=S + p$. Then, $p(t^\star) - p(\tilde{t}) \geq 0$ and
		\begin{equation}\label{eq:bounded_shift_gap}
			| \tilde{t} - t^\star | \leq \sqrt{\frac{  p(t^\star) - p(\tilde{t})}{c}} .
		\end{equation}
	\end{proposition}
	
	\begin{proof}
		Let us consider the difference in $S$ between $\tilde{t}$ and $t^\star$, i.e.,
		\begin{equation*}
			S(\tilde{t}) - S(t^\star) = h(\tilde{t}) - h(t^\star) - p(\tilde{t}) + p(t^\star).
		\end{equation*}
		Note that this is a positive quantity since $t^\star$ is the minimizer of $S$. However, since $\tilde{t}$ is assumed to be the global minimizer of $h$, we have $h(\tilde{t}) \leq h(t^\star)$, so that
		\begin{equation*}
			S(\tilde{t}) - S(t^\star) \leq p(t^\star) - p(\tilde{t}),
		\end{equation*}
		which also shows that $p(t^\star) - p(\tilde{t}) \geq 0 $. Since $S$ is strongly convex, we have $S(\tilde{t}) - S(t^\star) \geq c \, (\tilde{t}-t^\star)^2$, so that
		\begin{equation*}
			c \, (\tilde{t}-t^\star)^2 \leq p(t^\star) - p(\tilde{t}).
		\end{equation*}
		This leads to \eqref{eq:bounded_shift_gap}.
	\end{proof}
	
	We apply Proposition~\ref{prop:shift_local_bound} to our problem by using the functions $S = f^2 + g^2$ and $p = -2fg\cos(\varphi)$, where $\varphi$ is the continuous angle between the errors, see \eqref{eq:error_gauge}. It is easy to verify that $S$ is strongly convex if both $f$ and $g$ are. The proposition guarantees us that the shift in the location of the minimum will be bounded, but we can also argue that it must be relatively small. To see this, consider that the perturbation $p$ will be smallest around the minima of $f$ and $g$. Moreover, Lemma~\ref{lem:minimum_location_sum_convex} tells us that the minimizer of $S = f^2 + g^2$ is between those of $f$ and $g$. Therefore, the minimizer of the error gauge will be close to the minimizers of both true errors. We will demonstrate this in Section~\ref{sec:numerical_experiments}, see Figure~\ref{fig:casing_competition}.

	\subsection{The \TA}
	
	We now propose a new method that utilizes the error gauge as a stopping rule. We use here the shorthand $\Kdown(\bx)$ for a Kaczmarz down-sweep starting with $\bx$ and fixed parameter $0<\omega<2$. Likewise, $\Kup(\bup)$ denotes the twin Kaczmarz up-sweep. A possible implementation of the \TM\ in pseudocode is presented in Algorithm~1.

	\begin{algorithm}
		\caption{\textbf{\TA}}\label{alg:TA}
		\begin{algorithmic}[1]
			\Require $\bA$, $\bb$, $0 < \omega <2$, \code{maxits}
			\Output A regularized solution to \eqref{eq:Axb}.
			\State $\bx \leftarrow \zero$, \ $\bup \leftarrow \zero$
			\For{$k=1,\dots,$ \code{maxits}}
			\State $\bx \gets \Kdown(\bx)$
			\State $\bup \gets \Kup(\bup)$
			\IIf {$\| \bx - \bup\|$ is at a minimum} \textbf{break}, \EndIIf
			\EndFor
			\State \Return $\frac{1}{2} ( \bx + \bup)$.
		\end{algorithmic}
	\end{algorithm}
	
	The algorithm requires some maximum number of iterations \code{maxits}, which is simply a convenience. The output of the \TA\ is defined as the average of the stopped down- and up-sweeps. This is because we do not have any  preference for one over the other (this step can be skipped if multiple reconstructions are desirable).

	As with the statistical stopping rules, we have a sequence of positive real values whose minimum indicates where we should stop. Any strategy or method that is used to find the minimum for the statistical stopping rules can therefore be used for the error gauge. Typically, the functions occurring from the statistical stopping rules will not be smooth, and neither will the error gauge. As such, we opt to find the minimum by introducing a user-specified slack $s$, i.e.,
	a number of iterations to keep running after a (local) minimum has been found
	to accommodate any oscillations. The best approximation so far is stored.
	
	As an example, suppose that Kaczmarz is running and we come upon iteration $q$ and the slack is (re)initialized, and furthermore suppose that $e_{q+1}>e_q$, but $e_{q+2}<e_q$. In this case, the slack would reinitialize at $q+2$. We observe: a slack of size $s$ ignores oscillations of size $s-1$. Continuing the example, suppose that the smallest value of the error gauge has occurred at iteration $p$, then the algorithm will keep running at least until iteration $p+s$. If $e_p< e_j$ for $j=p+1,\ldots,p+s$, then the algorithm terminates at iteration $p+s$ and returns $\tfrac{1}{2}(\bx_p + \bup_p)$ as the reconstruction. Hence, we see that if $s$ is sufficiently large, the algorithm is guaranteed to find the global minimum of the error gauge.
	
	We use $s=7$, which is long enough for most oscillations we encountered in our test problems, but not so long that it wastes a lot of computational resources.

	\section{The \MSM}
	\label{sec:MSA}
	
	Up till now we employed our error gauge to select the best iteration while leaving the iterative method unaltered.
	Alternatively, we can adopt the error gauge to modify the method in such a way that it precludes the necessity of a stopping rule altogether.
	Specifically, we will exploit the error gauge to determine step lengths that eventually diminish,
	causing the method to converge to a good approximation of the noise-free solution.

	\subsection{Motivation for a new method}
	Let us define $\sd_k$ and $\sdt_k$ as the search directions in the
	down-sweep and up-sweep versions of Kaczmarz's method, respectively;
	that is, they formally satisfy
	\begin{equation}
		\begin{aligned}
			\sd_k &= \bA^T \bL^{-1} ( \bb - \bA \, \bx_k) \ ,\\[1mm]
			\sdt_k &= \bA^T \bL^{-T} ( \bb - \bA \, \bup_k) \ .
		\end{aligned}
	\end{equation}
	Recall that this definition includes the relaxation parameter $\omega \in (0,2)$ via the matrix $\bL$.
	Iteration $k{+}1$ of the down-sweep method is then given by $\bx_k + \sd_k$.
	A simple modification to the method allows for an iteration-dependent step size
	$\alpha_k$, so that $\bx_k + \alpha_k \sd_k$ is the next iteration.
	Similarly, we define $\bup_k + \beta_k \sdt_k$ as the
	next up-sweep iterate.
	Ideally, one wishes to plug these expressions into the exact error to
	find the optimal step sizes.
	However, the exact error is not available, so we use the error gauge instead.
	Hence, we aim to compute the step length parameters that solve the minimization problem
	\begin{equation}\label{eq:min_steps}
		\{\alpha_k,\beta_k\} := \arg \min_{\alpha,\beta} \, \tfrac{1}{2} \, \| \bx_k + \alpha \sd_k -
		\bup_k - \beta \, \sdt_k \|^2,
	\end{equation}
	which minimizes the error gauge of iteration $k+1$. We can find the minimum of \eqref{eq:min_steps} by setting the derivatives with respect to $\alpha$ and $\beta$ to zero, yielding
	\begin{equation}\label{eq:step_size_system}
		\left[ \begin{array}{cc} \| \sd_k \|^2 & -\sd_k^T \sdt_k \\[1mm]
			-\sd_k^T \sdt_k & \| \sdt_k \|^2 \end{array} \right]
		\left[ \begin{array}{c} \alpha_k \\ \beta_k \end{array} \right]
		= \left[ \begin{array}{r} - \sd_k^T (\bx_k - \bup_k ) \\[1mm]
			\sdt_k^T (\bx_k - \bup_k ) \end{array} \right].
	\end{equation}
	We solve this linear system for $\alpha_k$ and $\beta_k$ to obtain the step sizes.
	It is important to note that this system is nonsingular when the vectors
	$\sd_k$ and $\sdt_k$ are linearly independent.
	In this case, the $2 \times 2$ coefficient matrix in \eqref{eq:step_size_system}
	is symmetric positive definite.
	Of course, since the two search directions come from up-sweeps and down-sweeps,
	they will generally be linearly independent.
	If the search directions do happen to be dependent, we set $\alpha_k = 0$ and only solve for $\beta_k$.

	Because we are minimizing the distance between the up- and down-sweep iterates at every iteration by choosing suitable step sizes, the distance cannot increase. 
	Therefore the error gauge, which is this distance, will form a monotonically decreasing sequence.
	This immediately leads to the following result.
	
	\begin{proposition}\label{prop:converging_method}
		Suppose that $\sd_k$ and $\sdt_k$ are linearly independent for all $k$.
		By using the step sizes from \eqref{eq:step_size_system}, the error gauge
		$\| \bx_k - \bup_k \|$ converges to a local minimum.
	\end{proposition}
	
	\begin{proof}
		The error gauge is a monotonically decreasing sequence, while it is trivially bounded from below, i.e., $\| \bx_k - \bup_k \| \geq 0$.
	\end{proof}
	
	\begin{corollary}\label{cor:stop_orthogonality}
		If $\sd_k$ and $\sdt_k$ are linearly independent for all $k$,
		the step sizes $\alpha_k$ and $\beta_k$ converge to zero. Moreover, the asymptotic search directions $\sd$ and $\sdt$ are related to the limiting approximations $\bx$ and $\bup$ by
		\begin{equation}\label{eq:orthogonal_search}
			\sd^T (\bx - \bup) = 0 , \qquad \sdt^T (\bx - \bup) = 0.
		\end{equation}
	\end{corollary}
	
	\begin{proof}
		By assumption, $\sd_k$ and $\sdt_k$ are linearly independent, so that there is no linear combination resulting in the zero vector other than $\alpha_k = \beta_k = 0$. Proposition~\ref{prop:converging_method} asserts that the error gauge converges. Therefore, the step sizes must vanish as well, otherwise the error gauge would change.
		Second, since the system \eqref{eq:step_size_system} is symmetric positive definite, again by the assumption that $\sd_k$ and $\sdt_k$ are linearly independent, the zero solution can only occur when the right-hand side vanishes.
	\end{proof}
	
	Corollary~\ref{cor:stop_orthogonality} provides us the possibility to design a stopping criteria. According to \eqref{eq:orthogonal_search}, we can stop when
	\begin{equation}\label{eq:orthogonal_stop}
		\frac{|\sd^T_k (\bx_k - \bup_k)|}{ \| \sd_k \| \cdot \| \bx_k - \bup_k\| } \leq \varepsilon_1 \quad \text{and} \quad \frac{|\sdt^T_k (\bx_k - \bup_k)|}{ \| \sdt_k \| \cdot \| \bx_k - \bup_k\| } \leq \varepsilon_1,
	\end{equation}
	which is to say that the angles with the error gauge and the search directions are close enough to orthogonal.
	
	An additional stopping criterion also follows from Proposition~\ref{cor:stop_orthogonality}; a naive approach might be to stop simply when $|\alpha| + |\beta|$ falls below a specified threshold. This is guaranteed to happen since the method is convergent. It is worth pausing here to appreciate this remarkable result. The semi-convergence property has been completely circumvented, and yet, the \MSM\ will converge to an approximation of the semi-convergence point.
	
	\begin{keypoint}The \MSM\ is \emph{convergent}.
	\end{keypoint}
	
	The naive stopping criterion $|\alpha| + |\beta| \leq \varepsilon_2$ can be somewhat hard to interpret, as we have no control over the norm of the search directions $\sd_k$ and $\sdt_k$. Hence, we suggest an equivalent stopping criterion that has an easy interpretation as the relative change in the reconstructions. We propose to stop when
	\begin{equation}\label{eq:relative_change_stop}
		|\alpha_k| \frac{ \| \sd_k \| }{ \| \bx_k \| } + |\beta_k| \frac{ \| \sdt_k\| }{\| \bup_k \| } \leq  \varepsilon_2.
	\end{equation}
	Clearly, $|\alpha_k| \frac{ \| \sd_k \| }{ \| \bx_k \| } = \frac{\|\bx_{k+1}-\bx_k\|}{\|\bx_k\|}$, which is indeed the relative change in the down-sweep reconstruction. Hence, when the sum of the relative changes falls below the threshold $\varepsilon_2$, we stop. This is guaranteed to happen by Corollary~\ref{cor:stop_orthogonality}. To summarize: we test for both criteria \eqref{eq:orthogonal_stop} and \eqref{eq:relative_change_stop}, if either triggers, we stop.
	
	The threshold values $\varepsilon_1$ and $\varepsilon_2$ can be chosen as large as the minimum relative error that can be achieved. This will largely come down to experience and educated guesses. However, since the method is convergent, it is always possible to use the final up- and down-sweeps as input. With the \MSM, it is possible to run a certain number of iterations and inspect the reconstructions. If it is suspected that a better reconstruction is possible, the up- and down-sweeps can be reinserted as input. In the worst case, the images are unaltered. This is a distinct practical advantage that the \TA, or any statistical stopping rule for that matter, does not have.

	\subsection{The algorithm}
	
	We present here a possible implementation of the \MSM\ in pseudocode, see Algorithm~\ref{alg:MSA}. Similar to the \TA, the \MSA\ also computes two reconstructions. Since we have no bias towards one or the other, we again define the final reconstruction as the average of the two. Again, this step is not crucial for the algorithm, and it can be skipped if desired.

	\begin{algorithm}
		\caption{\textbf{\MSA}}\label{alg:MSA}
		\begin{algorithmic}[1]
			\Require $\bA$, $\bb$, $0 < \omega <2$, \code{maxits}, tolerances $\varepsilon_1$ and $\varepsilon_2$
			\Output A regularized solution to \eqref{eq:Axb}
			\State $\bx_0 \gets \Kdown(\zero)$, \ $\bup_0 \gets \Kup(\zero)$
			\State $\bx \leftarrow \bx_0$, \ $\bup \leftarrow \bup_0$
			\For{$k=1,\dots,$ \code{maxits}}
			\State $\sd \gets \Kdown(\bx) - \bx$
			\State $\sdt \gets \Kup(\bup) - \bup$
			\State Solve \eqref{eq:step_size_system} to determine $\alpha$ and $\beta$.
			\State \code{cond1} = $\{ \frac{|\sd^T (\bx - \bup)|}{\|\sd\| \, \|\bx - \bup\|} \leq \varepsilon_1$ \textbf{and}  $\frac{|\sdt^T (\bx - \bup)|}{\|\sdt\| \, \|\bx - \bup\|} \leq \varepsilon_1 \}$
			\State \code{cond2} = $\{|\alpha| \frac{ \| \sd \| }{ \| \bx \| } + |\beta| \frac{ \|\sdt\| }{\| \bup \| } \leq \varepsilon_2 \}$
			\IIf{ \code{cond1} \textbf{or} \code{cond2}  } \textbf{break},\EndIIf
			\State $\bx \gets \bx + \alpha \sd$
			\State $\bup \gets \bup + \beta \, \sdt$
			\EndFor
			\State \Return $\frac{1}{2} ( \bx + \bup)$.
		\end{algorithmic}
	\end{algorithm}
	
	The \MSA\ requires two different starting vectors;
	if we start with the same vector, the initial error gauge would be zero, causing the step sizes to come out zero. Consequently, the algorithm would exit immediately.
	Furthermore, according to \cite{elfving_2014},
	the starting vector should lie in the row space of $\bA$.
	Our approach is simply to use a single down-sweep and up-sweep starting from the zero vector.
	These vectors are different and lie in the correct space.

	\subsection{Computational cost}\label{sec:cost}
	
	We now turn to the computational cost of the \MSM .
	At this point, it is convenient to introduce a \textit{work unit}: a certain number of operations
	so that we can easily compare the cost of the various methods.
	The most convenient work unit in our context is a single Kaczmarz sweep.
	For X-ray tomography problems, the average number of nonzero elements
	in a row of $\bA$ is $\sqrt{n}$ for a 2D problem
	and $\sqrt[3]{n}$ for a 3D problem. Carefully going through the operations in \eqref{eq:ART} reveals that there are
	$4 m \sqrt{n}$ or $4 m \sqrt[3]{n}$ operations in a sweep.
	We omit operations that scale as constants as they will be negligible.
	Note that the cost of one matrix-vector multiplication is half a work unit.
	
	\begin{table}
		\renewcommand{\arraystretch}{1.3}
		\centering
		\caption{Work load of the various methods for a 2D problem; for a 3D problem replace
			$\sqrt{n}$ with $\sqrt[3]{n}$.
			A work unit is defined as the work in a single Kaczmarz sweep.}\label{tab:work_loads}
		\begin{tabular}{l|ll} \hline
			Method & Operations & Work units  \\ \hline
			Standard Kaczmarz & $\ph{1}4 m \sqrt{n}$ & 1 \\[0.5mm]
			Idem + trace-estimate stopping rules & $10m \sqrt{n} + 2 n + 3m$ &
			$\tfrac{5}{2} + \tfrac{1}{2} \frac{\sqrt{n}}{m} + \tfrac{3}{4} \frac{1}{\sqrt{n}}$ \\[0.5mm]
			\TM & $\ph{1}8 m \sqrt{n} + 3n$ & $2 + \tfrac{3}{4} \frac{\sqrt{n}}{m}$\\[0.5mm]
			\MSM & $\ph{1}8 m \sqrt{n} + 11n$ & $2 + \tfrac{11}{4} \frac{\sqrt{n}}{m}$ 
			\\[0.5mm] \hline
		\end{tabular}
	\end{table}

	For each method, we examine the total cost and express it in terms of work units in Table~\ref{tab:work_loads}.
	The \TM\ requires two Kaczmarz sweeps together with determination
	of the difference between the two iterates.
	The \MSM\ requires two Kaczmarz sweeps and the solution of the system \eqref{eq:step_size_system}.
	All the stopping rules (see Section~\ref{sec:numerical_experiments} for details)
	require a trace estimate which means that, per iteration, the additional amount of work is one Kaczmarz sweep, the determination of the residual and its norm, and an inner product.
	There are a small number of additional operations which we ignore here.
	
	Naturally, as $m$ and $n$ grow, the cost of each method becomes dominated
	by the cost of the Kaczmarz sweeps and the determination of the residual, if needed.
	The cost of inner products are evidently negligible in the larger scheme of things.
	Our proposed methods do not require the residual, so that their asymptotic cost is 2 work units
	per iteration. The stopping rules have an asymptotic cost of $\tfrac{5}{2}$ work units per iteration.
	Our methods are therefore slightly cheaper for large systems.
	As an example, for a small $128 \times 128$ image we have a system size of $n = 128^2 = 16384$ and $m \approx 1.2 \, n$.
	We should remark that this choice of $m$ is entirely arbitrary. In practice, a whole range of $m$ is used, from vastly underdetermined systems to extremely overdetermined.
	This amounts to a cost of about $2.005$ work units per iteration for the \TM\ and $2.02$ work units per iteration for the \MSM .

	In terms of storage, the \TA\ as we use it costs the same as the statistical stopping rules, while the \MSA\ is more expensive. The statistical stopping rules require three vectors of length $n$ to be stored: the actual image, the initial random Gau{\ss}ian noise vector and the Kaczmarz iterate that uses the random noise vector as its initial guess. The \MSA\ requires the search vectors and image vectors to be stored, which means it requires four $\mathbb{R}^n$-vectors of storage. The \TA\ with slack period requires storing the best reconstruction so far, together with two image vectors, resulting in 3 $\mathbb{R}^n$-vectors -- exactly the same as the statistical stopping rules. However, we must point out that it is possible in the \TA\ to employ other methods of detecting the minimum in the error gauge. The method presented here was chosen for performance reasons; it is guaranteed to find the global minimum provided the slack is sufficiently large. It is possible to use local estimators to find the minimum, which would bring the memory requirements down to only two vectors of length $n$.

	\section{Numerical experiments\label{sec:numerical_experiments}}
	To conduct our numerical experiments, we use the \textsc{AIR Tools II} package for MATLAB which contains various codes for the creation and solution of tomographic problems \cite{Hansen2018}. The package also contains a function \code{phantomgallery} that creates various phantoms with different features. We use the parallel beam set-up with $128 \times 128$ pixels per phantom, projection angles $0^\circ,1.5^\circ,3^\circ,\dots, 178.5^\circ$ and $\code{round}(\sqrt{2}\cdot 128) = 181$ rays per projection.
	Furthermore, as a tolerance for the \MSA\ we use $\code{tol1} = \code{tol2} = 10^{-4}$.
	
	\begin{keypoint}
		We will use $\bx_0 = \zero$ throughout this work, unless mentioned otherwise.
		This choice simplifies the expressions somewhat, but more importantly it serves as a good initial guess for noisy inverse problems.
	\end{keypoint}
	
	To simulate noise, we add white Gau{\ss}ian noise $\delta \bb \sim \mathcal{N}(\boldsymbol{0}, \sigma^2 \boldsymbol{I}_m)$ scaled such that we can specify the expected relative noise level to be $\eta$, i.e., 
	\begin{equation}
		\eta^2 =\frac{ \mathbb{E} \!\left(  \|\delta \bb\|^2 \right) }{ \|\bb\ex \|^2}  =
		\frac{m\,\sigma^2}{ \|\bb\ex \|^2} \ .
	\end{equation}
	
	\begin{figure}
		\centering
		\includegraphics[width=.9\textwidth]{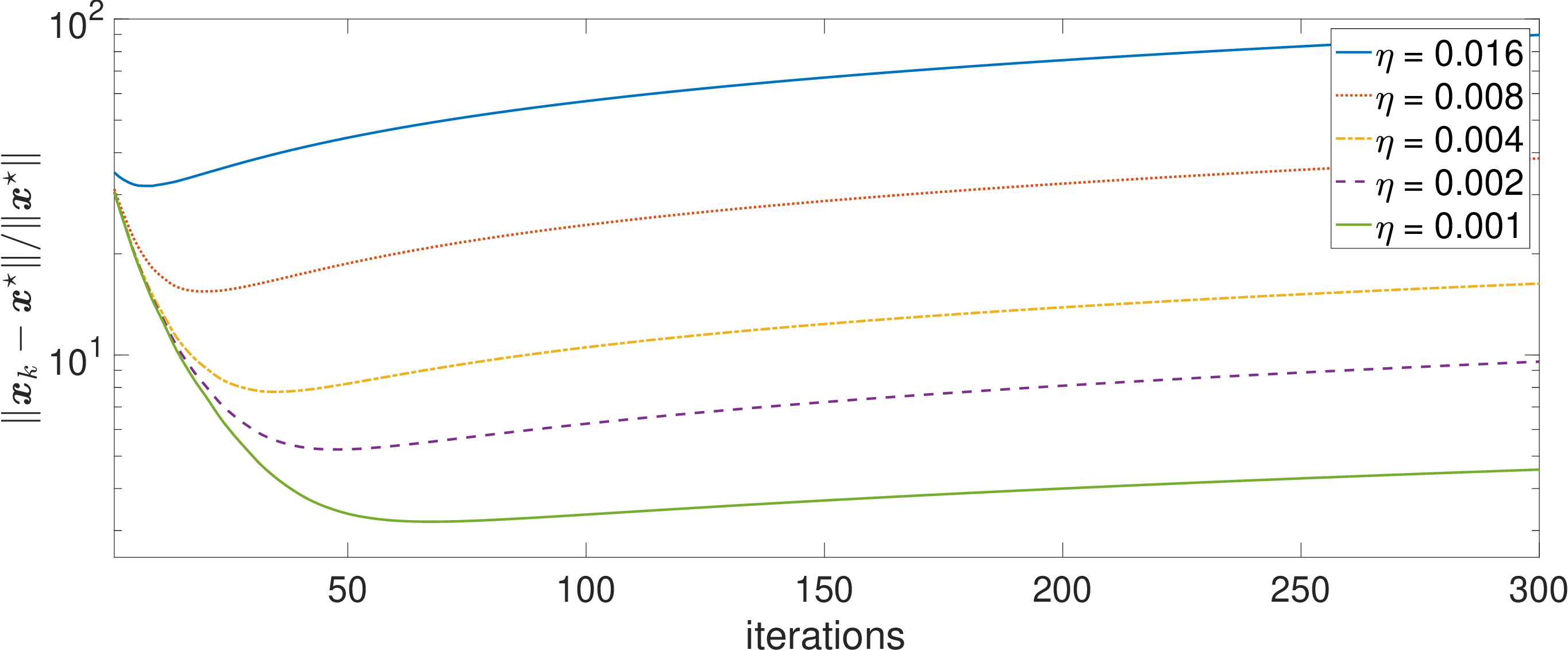}
		\caption{Error histories for Kaczmarz's method
			with $\omega=1$ and different relative noise levels $\eta$.}
		\label{fig:effect_of_noise}
	\end{figure}
	
	To demonstrate the effect of noise on the reconstruction, Figure~\ref{fig:effect_of_noise} shows error histories of the relative error $\| \bx_k - \bx^\star \| \ / \, \| \bx^\star \|$ for various noise levels.
	Our phantom of choice is the \code{grains} phantom, which simulates the polycrystalline structure found in many metals, rocks, and bones.
	We used the standard Kaczmarz (down-sweep) method with $\omega = 1$. We see that as $\eta$ increases the whole curve moves up and the minimum becomes less flat.
	
	For the lowest noise level, $\eta = 10^{-3}$, any iteration between $k=70$ and 100 gives almost the same relative error. However, for the highest noise level it is more critical to find the right number of iterations. The general trend is clear: for higher noise levels the stopping rule needs to be more accurate.
	
	\subsection{Casing the competition}\label{sec:casing}
	
	As the closest competitors of our proposed algorithms, we consider
	the standard Kaczmarz algorithm with either of the three statistical stopping rules
	from Section~\ref{sec:stoprules}: UPRE, GCV, and CDP\@.
	Our experience is that the performance of these stopping rules for Kaczmarz's algorithm is generally quite poor, especially for high noise levels, and to demonstrate this we show a representative error history for the \code{grains} phantom in Figure~\ref{fig:casing_competition}.

	\begin{figure}
		\centering
		\includegraphics[width=.9\textwidth]{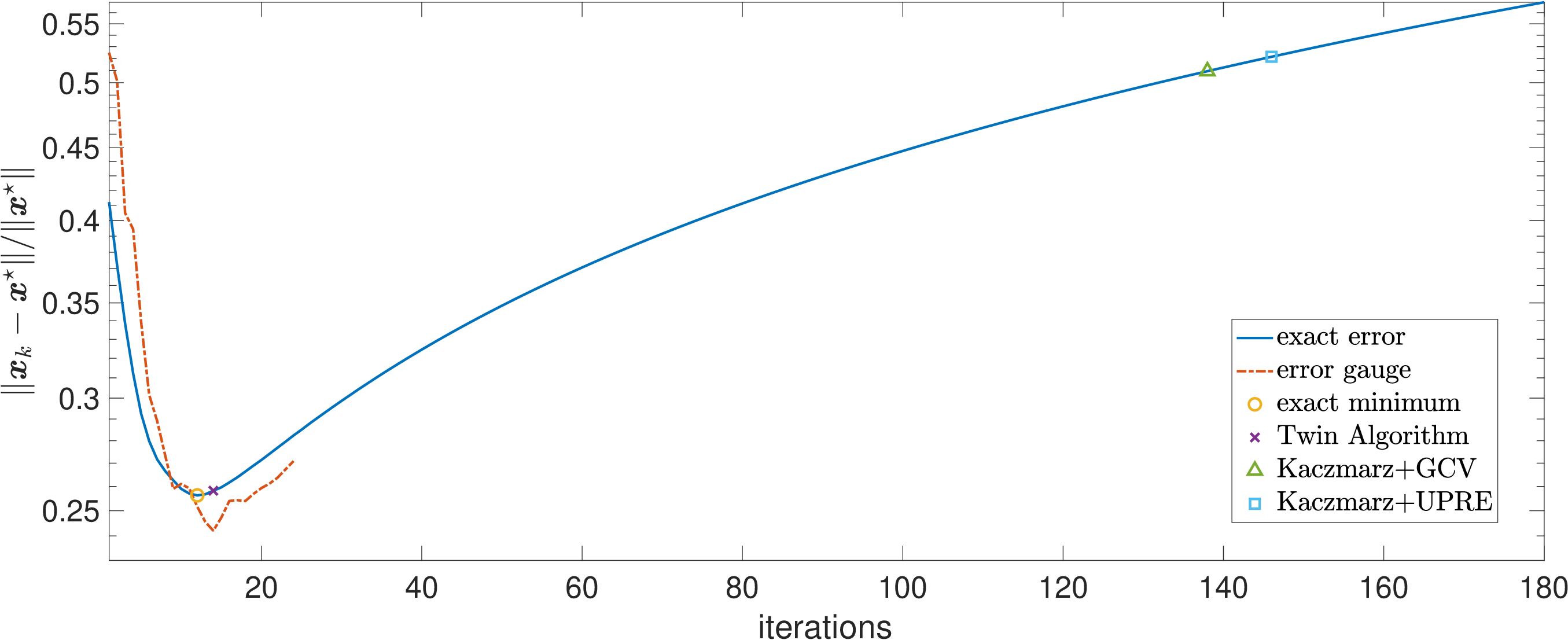}
		\caption{Comparing the proposed \TA\ with its built-in stopping rule to statistical ones applied to standard down-sweep Kaczmarz. Both the exact error of the standard Kaczmarz algorithm and the error gauge are scaled with the norm of the noise-free solution, i.e.,
			we show $\| \bx_k - \bx\ex \| \, / \, \| \bx^\star \|$ and $\| \bx_k - \bup_k \| \, / \, \| \bx^\star \|$, respectively.
		}
		\label{fig:casing_competition}
	\end{figure}

	It is clear that all three statistical stopping rules overshoot the mark by quite a margin. GCV and UPRE overshoot by roughly 100 iterations, while CDP did not stop for \code{maxits} $= 300$. As already mentioned in Section~\ref{sec:intro}, this may not matter very much for the simultaneous iterative methods where the minimum is very flat. For Kaczmarz's algorithm, on the other hand, there is a significant difference. For the current example, our error gauge stops within two iterations of the minimum and produces an image that is roughly 60\% better compared to the statistical stopping rules.

	As we have emphasised before, the output of the \TA\ is \textit{not} equal to any iteration of Kaczmarz's method. For ease of presentation, we have indicated in Figure~\ref{fig:casing_competition} the iteration at which the \TA\ stops. The error of the output -- the average of the up- and down-sweeps -- typically has a smaller error and lies below the exact error curve. The outputs of Kaczmarz equipped with GCV and UPRE are in fact on the curve, and this would also be the case for CDP be if it had stopped within \code{maxits} iterations. Instead of comparing the proposed algorithms with statistical stopping rules, we therefore from now on compare with the exact minimum.
	
	We will suppose that we have an oracle\footnote{The oracle is a concept borrowed from computational complexity theory \cite{papadimitriou}: oracle machines ``are machines that are given access to an ``oracle" that can magically solve the decision problem for some language". Here, we will take the oracle to magically provide the exact error of a reconstruction. }
	that can tell you the exact error of a reconstruction, but importantly, not anything else. Having access to the exact error allows one to pick the best iteration from the sequence of reconstructions generated by an iterative method. This is what we will compare our algorithms with.

	\subsection{Effect of the relaxation parameter}\label{sec:relax}

	Throughout this work, we have ignored the choice of a relaxation parameter. However, for many applications this is a crucial point and it is a fair question to ask whether or not the error gauge works for relaxation parameters other than $\omega = 1$. To demonstrate the affirmative, we present Figure~\ref{fig:varying_omega}, which is produced using the test problem as previously defined with the \code{grains} phantom and a relative noise level of $8 \cdot 10^{-3}$.
	
	\begin{figure}
		\centering
		\includegraphics[width=.9\textwidth]{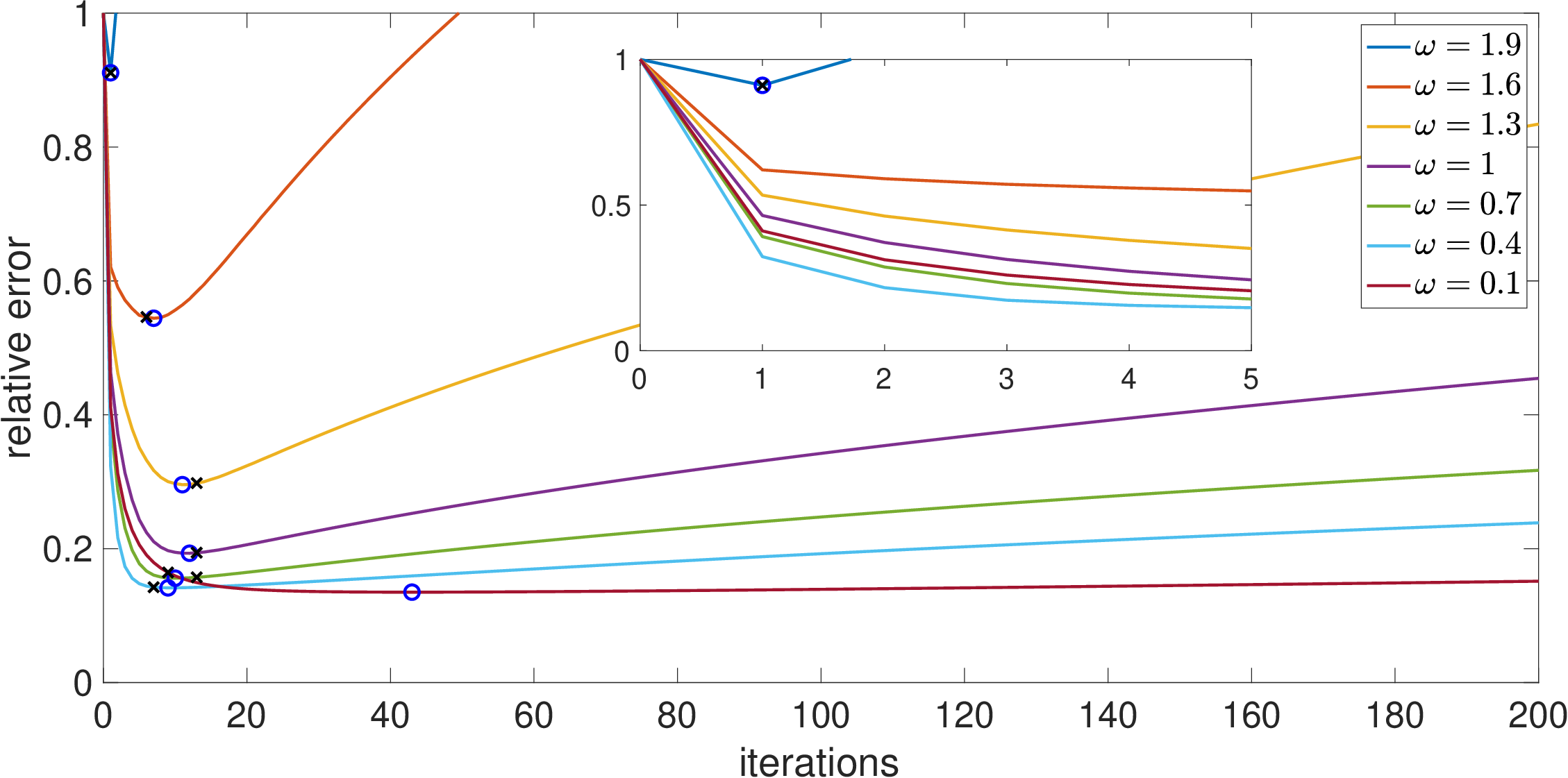}
		\caption{Varying the relaxation parameter $\omega$. The blue circles indicate the oracle's choice, i.e., the minimum of the error curve, while the crosses indicate the iteration picked out by the error gauge. The inset shows the first 5 iterations. The \code{threephases} phantom was used to test.
		}
		\label{fig:varying_omega}
	\end{figure}
	
	The figure shows that for a range of values $0< \omega < 2$, the error gauge picks out an iteration close to the oracle's choice, i.e., the iteration with the smallest error. The error gauge typically produces a good reconstruction, only being a couple of iterations off in most cases. 
	
	An interesting point to note about Figure~\ref{fig:varying_omega} is that a smaller relaxation parameter produces a smaller minimal error.
	This might be somewhat surprising, as asymptotic convergence theory shows that the optimal relaxation parameter satisfies $\omega>1$: it is known that Kaczmarz is equivalent to Successive Over-Relaxation applied to the system $\bA \bA^T \by = \bb$ with $\bx = \bA^T \by$ \cite{bjorck_elfving}, while it can be shown that $\omega>1$ results in the asymptotically optimal convergence rate for Successive Over-Relaxation \cite{saad}. However, we must note that we are optimizing different objectives: in the asymptotic case we optimize the convergence rate, while in the noisy case we are optimizing the minimal error at the semi-convergence point.
	The observed behavior of a smaller relaxation parameter producing a smaller error is quite consistent for the phantoms we tested with. Moreover, typically a smaller relaxation parameter also leads to slower convergence. As a compromise between computation time and performance, we suggest a relaxation parameter in the range of  $0.4 \leq \omega \leq 0.7$, which produces good results for all phantoms in a decent time.

	We must also note that the error gauge does not seem to work quite as well for very small relaxation parameters. However, as we have already pointed out, these cases are likely to be avoided for reasons of speed. In Figure~\ref{fig:varying_omega}, the error gauge only has difficulties with $\omega = 0.1$.
	
	\subsection{Variants of row-by-row Kaczmarz methods}
	
	We will also briefly demonstrate that our error gauge approach also works on some variants of Kaczmarz's method. Specifically, the symmetric Kazcmarz method and a special family of block-sequential Kaczmarz methods.
	
	The symmetric Kaczmarz method consists of performing an up-sweep after a down-sweep or the other way around. The iteration matrices are given by $\bG^T\bG$ and $\bG\bG^T$, which therefore share eigenvalues -- the squared singular values of $\bG$ -- but not eigenvectors. Indeed, the eigenvectors of $\bG^T\bG$ are the right-singular vectors of $\bG$, while for $\bG\bG^T$ we have the left-singular vectors as eigenvectors. We can therefore again construct an error gauge by comparing the up-down version with the down-up version.
	
	Following Elfving and Nikazad, we refer to a collection of rows of $A$ as a block,
	and each block is processed by computing the pseudoinverse \cite{elfving}. From the fact that Kaczmarz's method computes a minimum-norm solution to \eqref{eq:augmented_system}, it can be seen that Kaczmarz's method actually computes the pseudoinverse when the rows of $\bA$ are orthogonal and a starting guess $\bx_0 = \zero$ is used. Thus, if all blocks consist of sets of orthogonal rows, each can be treated by using Kaczmarz's method. The conclusion is that a block-sequential method with blocks of orthogonal rows is equivalent to a row-by-row Kaczmarz method. Moreover, when such blocks consist of structurally orthogonal rows, meaning each term in the inner product vanishes, the result can be computed in a parallelized way. S{\o}rensen and Hansen experimentally confirmed that such block-sequential methods perform very well and can be efficiently implemented in parallel \cite{sorensen}.

	\begin{figure}
		\centering
		\includegraphics[width=.9\textwidth]{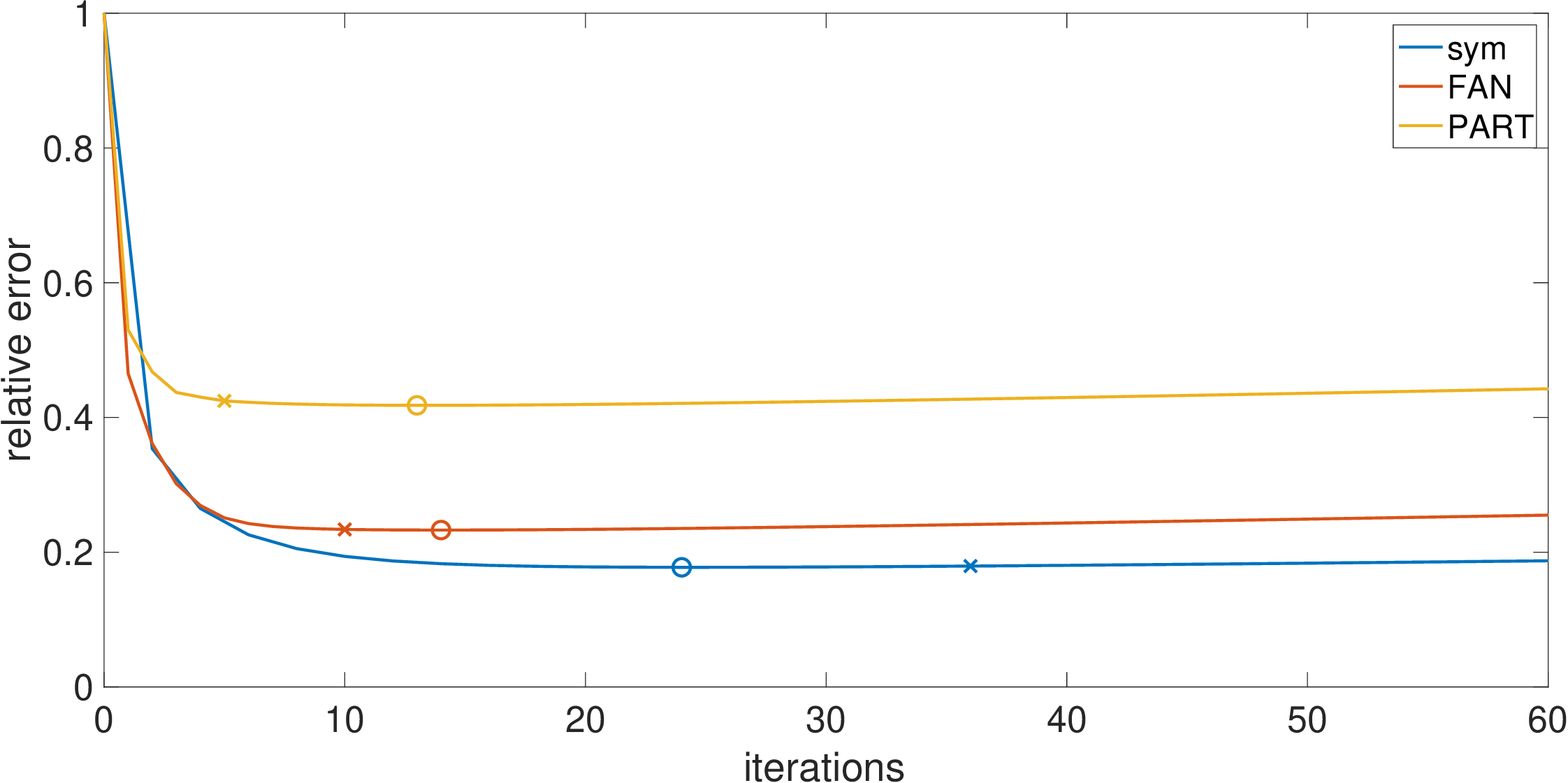}
		\caption{Error histories for the \code{shepplogan} phantom for various Kaczmarz methods that can be implemented as a row-by-row method. All methods use $\omega = 0.4$. Crosses indicate the error gauge stopping point, while circles indicate the minima.
		}
		\label{fig:family_of_methods}
	\end{figure}
	
	The results are plotted in Figure~\ref{fig:family_of_methods}. The methods are displayed in terms of iterations of Kaczmarz's method; note that one iteration of symmetric Kaczmarz takes two standard Kaczmarz iterations. The error gauge does not perform quite as well for the symmetric Kaczmarz method.
	
	The PART (parallel ART) blocking strategy has been first proposed by Gordon \cite{gordon}, who uses blocks whose rows correspond to parallel beams, ensuring each block consists of structurally orthogonal rows. Structural orthogonality between rows roughly corresponds to rays that do not intersect inside the measurement volume, hence it is not very difficult to construct other parallelizable versions of Kaczmarz. To demonstrate that this is indeed fairly easy, we have constructed a variant of PART that groups rays into fans with their shared intersection point outside the measurement volume. As a consequence, each block consists of structurally orthogonal rows. We have labeled this ``FAN'' in Figure~\ref{fig:family_of_methods}. All methods were run with a relaxation parameter of $\omega = 0.4$.

	\subsection{Randomized Kaczmarz methods}
	\label{sec:rkm}
	The idea of the error gauge does not seem to easily combine with randomized methods. We believe this is due to the fact that the error gauge requires a fixed set of eigenvectors to work. We can arbitrarily call $m$ updates an iteration and set up the iteration matrix, however, it would clearly change with every iteration.
	To test randomized Kaczmarz in conjunction with the error gauge, we employ a popular randomization: a random shuffle of the rows each sweep.
	The popularity stems from the fact that each row is used exactly once in an iteration, thereby making sure all of the data is used. The shuffle can be generated by, e.g., the Fisher--Yates algorithm \cite[Vol.~2, Sec.~3.4.2, Algo.~P]{knuth}. The iterate $\bx_k$ is computed using the shuffled ordering, while $\bxu_k$ is computed using the reversed shuffled ordering.
	The results are plotted in Figure~\ref{fig:rand_eg_fail}. A fixed ordering is also shown in the figure for comparison.
	
	\begin{figure}
		\centering
		\includegraphics[width=.9\textwidth]{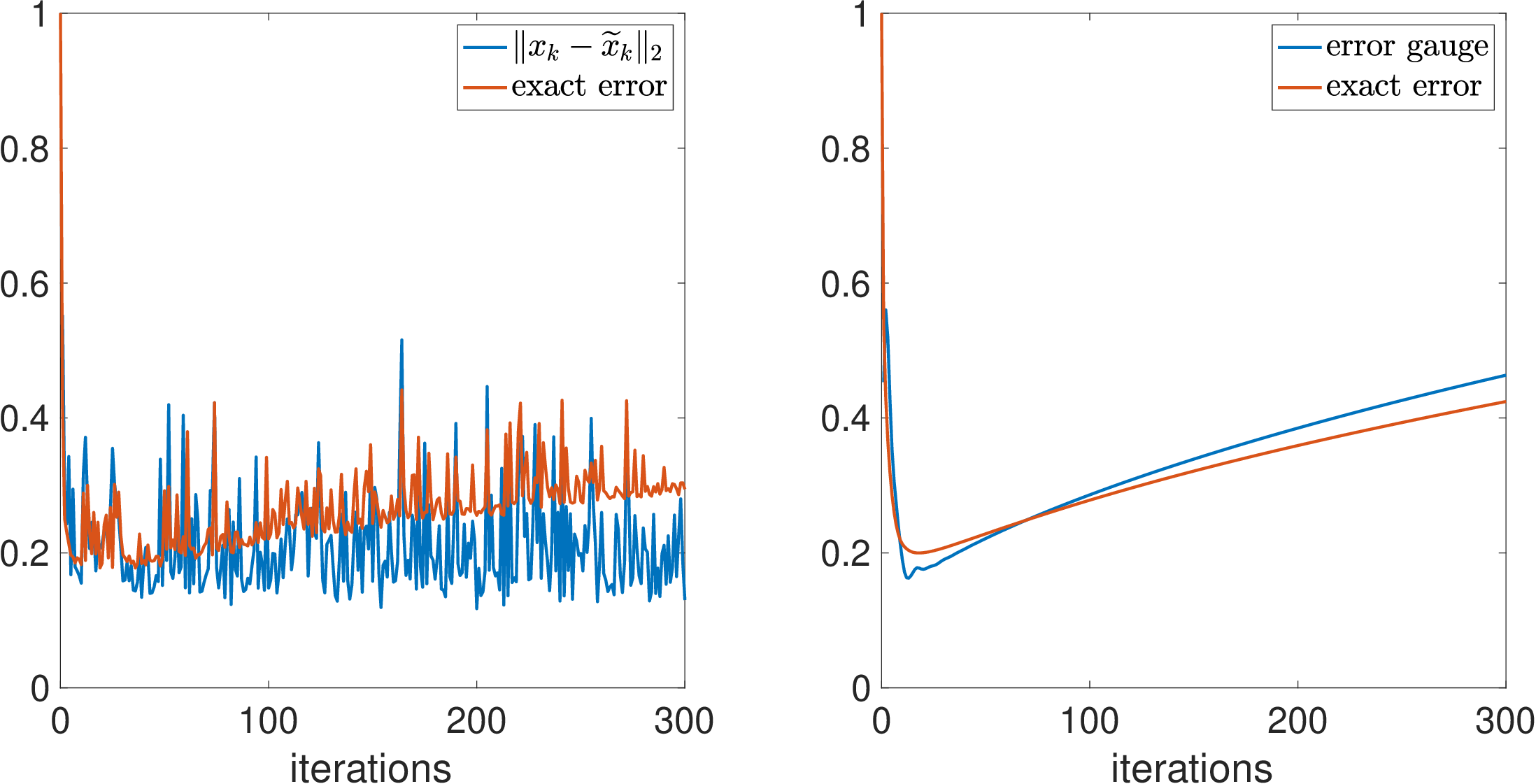}
		\caption{Results for the \code{shepplogan} phantom. Left: $\| \bx_k - \bxu_k \|_2$ and exact error for randomly shuffled Kaczmarz. Right: $\| \bx_k - \bxu_k \|_2$ -- the error gauge -- and exact error for cyclic Kaczmarz.}
		\label{fig:rand_eg_fail}
	\end{figure}
	
	It is interesting to note that many of the peaks occurring in the exact error are reflected in $\| \bx_k - \bxu_k \|_2$. However, most notably should be the fact that $\| \bx_k - \bxu_k \|_2$ does not show an increasing trend after the semi-convergence point. By contrast, in the cyclic case $\| \bx_k - \bxu_k \|_2$ follows the behavior of the exact error rather well. We must conclude that $\| \bx_k - \bxu_k \|_2$ is not an error gauge for randomized methods, and depends on a fixed ordering. Extending the error gauge to work with randomized methods is not a simple modification.

	Any randomized Kaczmarz method is going to give a distribution of outcomes for a fixed right-hand side. If we consider the infinite sequence of iterates as the outcome, any instance of a randomized Kaczmarz method will have a vanishing probability mass. At the same time, we should only compare events that have a nonzero probability. We can, for instance, compare randomized methods with cyclic methods by their average behavior. Alternatively, we can estimate the probability that an outcome of a randomized method produces a better result than a cyclic method. It turns out that for the CT problem, at least for our test problem, simply choosing a good relaxation parameter $\omega$ makes the cyclic method better than the randomized method in these senses.
	
	\begin{figure}
		\centering
		\includegraphics[width=.9\textwidth]{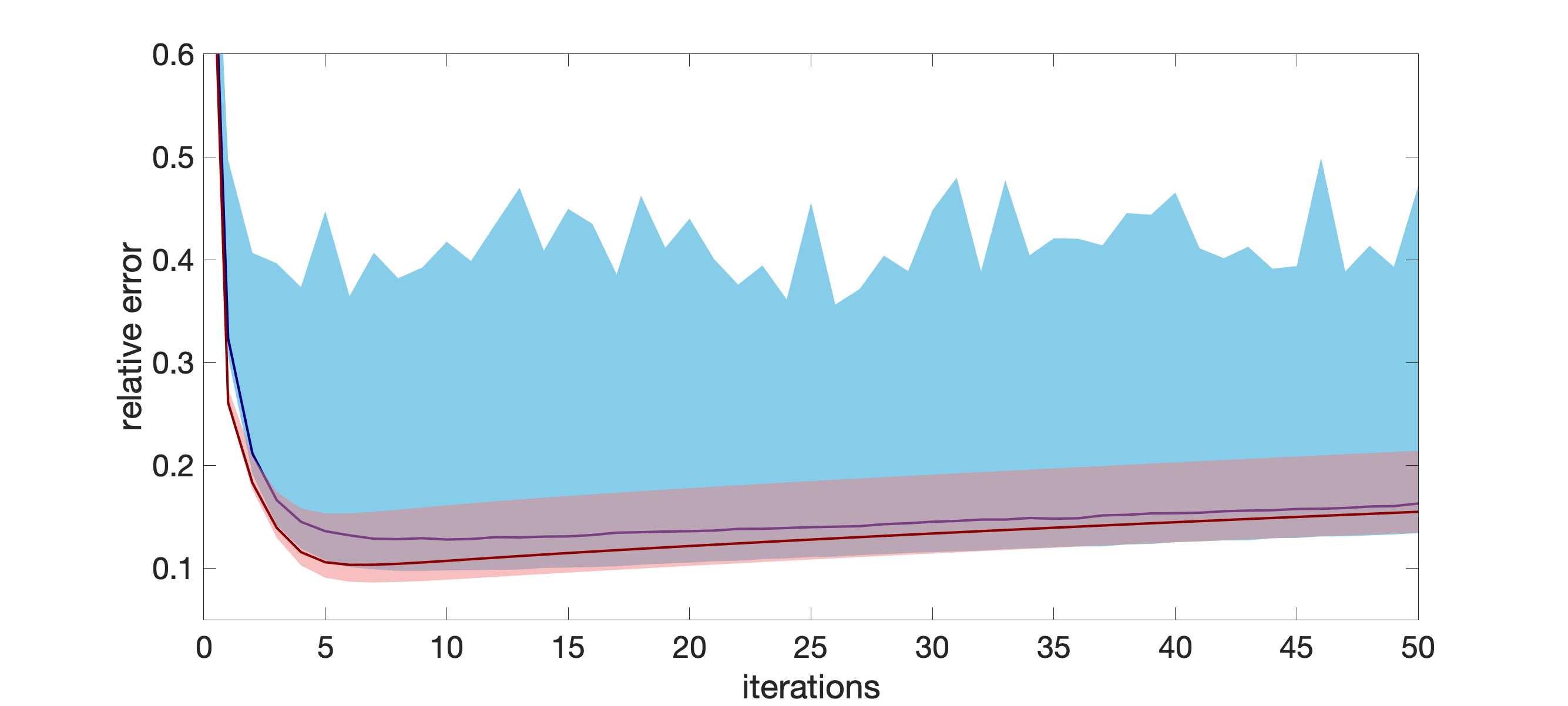}
		\caption{Average performance of cyclic (red line) and randomized methods (blue line) for $\omega = 0.4$. The shaded regions indicate the range of outcomes. Tests performed using the \code{shepplogan} phantom}
		\label{fig:rand_vs_cyclic}
	\end{figure}
	
	For randomized Kaczmarz, we have observed the general trend that a smaller relaxation parameter leads to a smaller minimal error.  Hence, the relaxation parameter $\omega = 0.4$, which is on the lower end of range we suggested in Section~\ref{sec:relax}, is also a good choice for randomized Kaczmarz. 
	In Figure~\ref{fig:rand_vs_cyclic}, we have performed 3000 runs of our test problem.
	For each run, Gau{\ss}ian noise was generated with a relative noise level of $8 \cdot 10^{-3}$. The results are therefore both averaged over the noise instances and the random orderings.
	
	As should be clear from the figure, for this example the cyclic Kaczmarz method converges faster and provides a smaller error on average. In 1891 out of 3000 cases the cyclic method produces a lower minimal error, which corresponds to roughly $63\%$ of the cases. This demonstrates that randomized methods do not always outperform cyclic Kaczmarz.
	
	It is, of course, possible that different randomized methods do perform better in some sense; how to choose the row-selection probabilities is a whole topic in and of itself. But then again, it is also possible to find better row orderings for the cyclic methods. Pursuing this quickly leads to a completely different topic that is outside the scope of the current work: we are concerned here with demonstrating the effectiveness of our error gauge.

	\subsection{Haunted house:\ a selection of phantoms}\label{sec:haunted_house}
	
	From now on, we use the relative noise level $\eta = 8 \cdot 10^{-3}$ in our experiments, as this may be seen as realistic; it is also the second-largest noise level from Figure~\ref{fig:effect_of_noise}. We also fix $\omega = 0.7$, which is in the upper end of the range we suggested in Section~\ref{sec:relax}.
	To illustrate the point of Section~\ref{sec:casing}, we plot the results of the various algorithms under consideration for
	two phantoms from \code{phantomgallery}, namely, the already-mentioned \code{grains} phantom and \code{shepplogan} which implements the Shepp--Logan phantom.
	The corresponding reconstructions are shown in Figure~\ref{fig:reconstructions}.
	Both phantoms have pixel values between 0 and 1.
	When we compute the relative error with respect to each phantom we use the solutions
	as produced by the algorithms, with some negative pixels and some pixels greater than one.
	The figures, on the other hand, show the reconstructions with the pixel values
	limited to the range $[0,1]$. For both phantoms the \MSA\ produces the best reconstruction.
	
	\begin{figure}
		\centering
		\includegraphics[width = .7\textwidth]{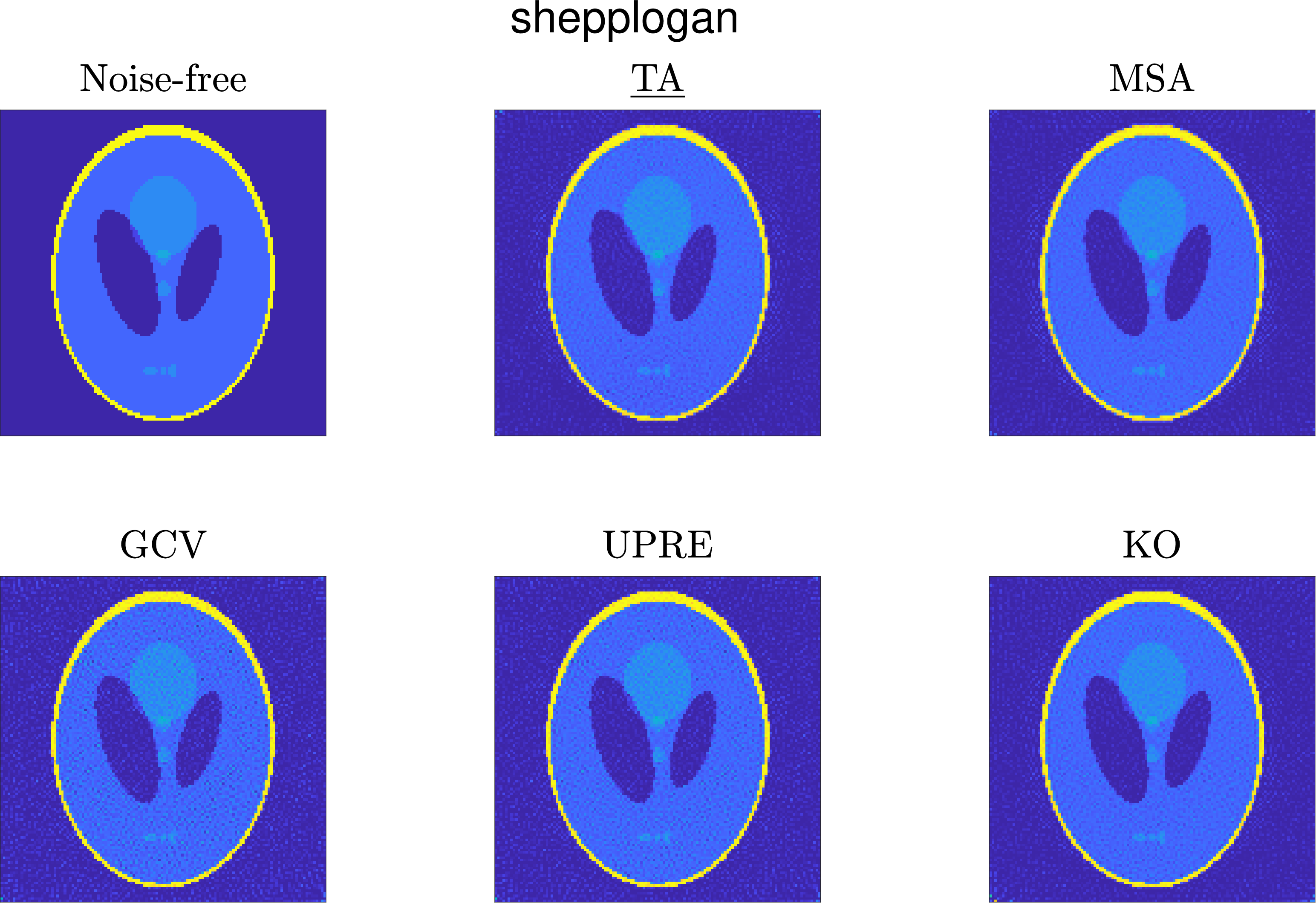} \\[8mm]
		\includegraphics[width = .7\textwidth]{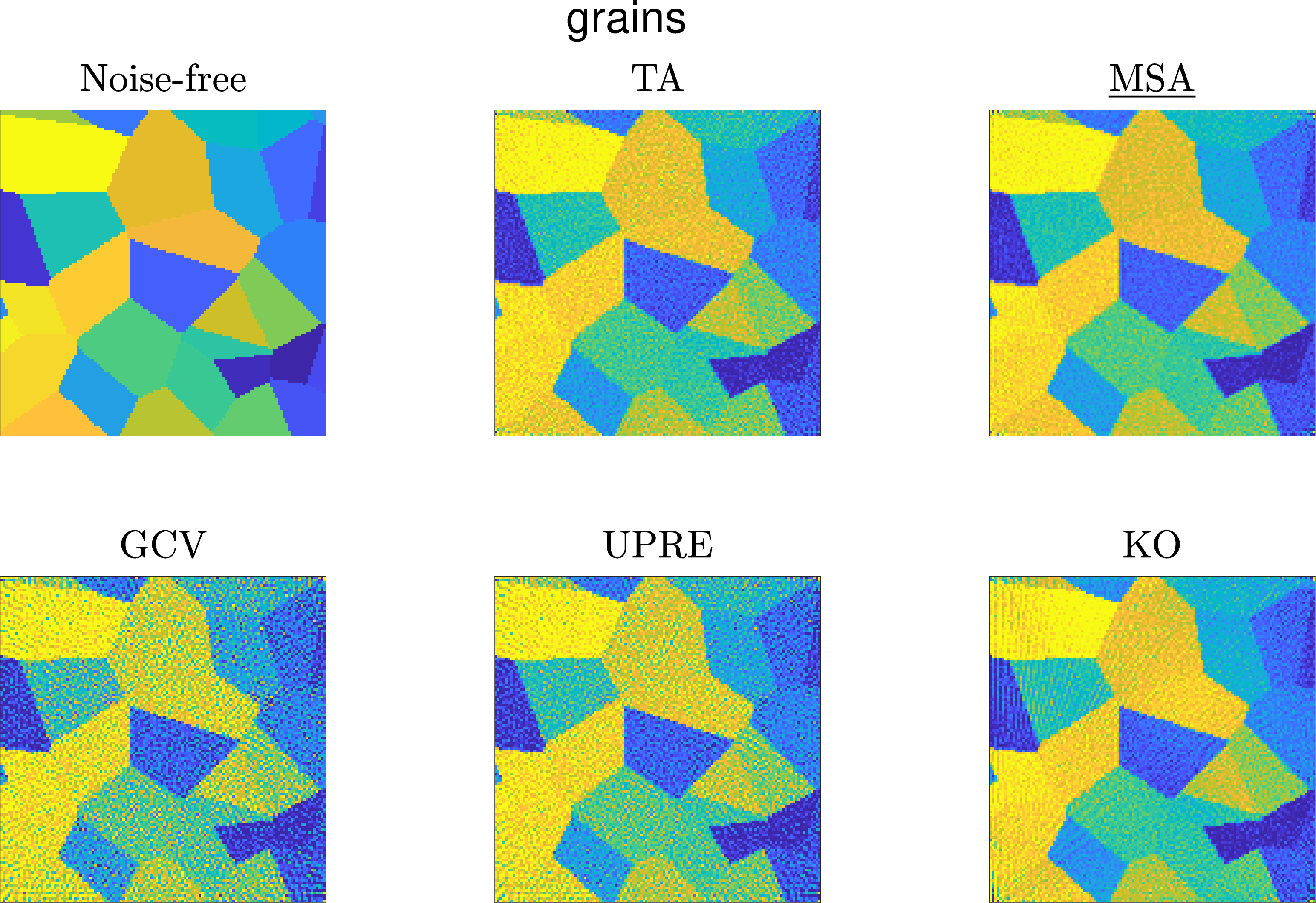}
		\caption{Numerical experiments with the \code{shepplogan} phantom (top) and the \code{grains} (bottom) phantom from \textsc{AIR Tools II}.
			The \TA\ produces the best result for \code{shepplogan}, while the \MSA\ produces minimum error for \code{grains}. When displaying the reconstructions the intensity is limited to $[0,1]$.}\label{fig:reconstructions}
	\end{figure}
	
	First consider the results for the \code{shepplogan} phantom. The \TA\ produces the best reconstruction: it is the least grainy picture,
	while the small details can be clearly distinguished. Both GCV and UPRE produce
	a very grainy image and the smallest details are harder to identify. For reference, we have also included a result using the oracle as a stopping rule. We refer to this algorithm as ``\KO '', which is indicated by KO in the figure.
	
	The results for the \code{grains} phantom are slightly more interesting to examine. This phantom consists of a collection of piecewise constant regions (which are Voronoi regions belonging to a random collection of points). The relative performance of the algorithms can be visually evaluated by looking especially at the contrast, e.g., between two neighboring regions that have close intensity values. Whether or not one can distinguish two such regions is a matter of opinion, but generally speaking it does appear that the \MSA\ produces the best results as well as the least grainy picture. GCV and UPRE perform particularly badly, it seems, where neighboring regions are sometimes very hard, if not impossible, to discern.
	
	To demonstrate that the observed behavior is not an oddity or statistical fluke, we run the reconstructions of the \code{grains} phantom 1000 times
	and keep track of the relative errors. From this point on, we compare our algorithms only to \KO,
	as the statistical stopping rules can never do better.

	The results are shown in the histograms in Figure~\ref{fig:error_histogram}.
	The histograms show several interesting features.
	First off, the \MSA\ is overall best, with the lowest average error and the smallest spread.
	In fact, the outcomes from this algorithm has very little or no overlap with the other algorithms.
	The \TA\ gives a slightly better average performance than \KO,
	though their histograms overlap almost completely. \KO\ has the largest spread, and is on average slightly worse than the \TA .
	At first, it may seem odd that the \TA\ produces a better result on average than the \KO.
	Yet, we should recall that the output of the \TA\ is the average of the up-sweeps and down-sweep iterates, while the \KO\ algorithm is associated with the down-sweep iterates only.
	Evidently the averaging often gives a better result than a down-sweep or up-sweep algorithm separately.

	\begin{figure}
		\centering
		\includegraphics[width=.9\textwidth]{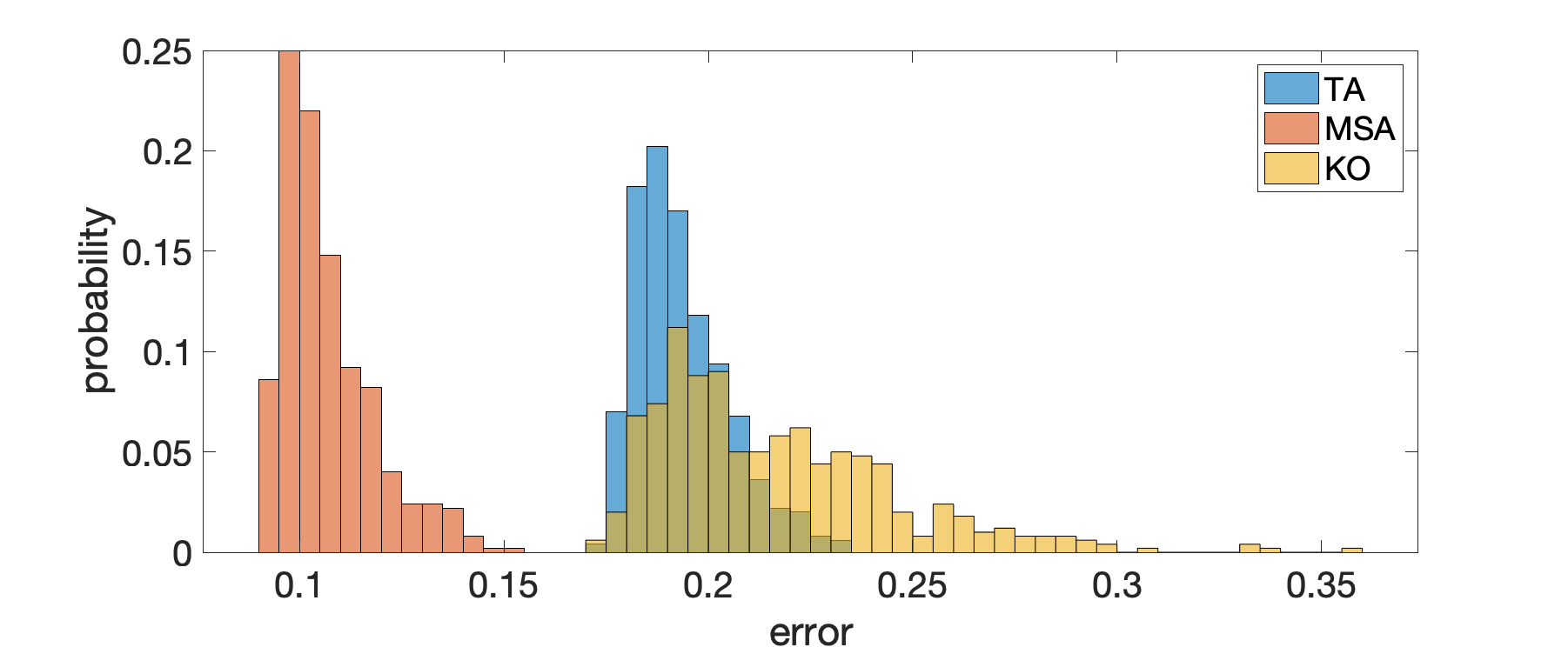}
		\caption{Histogram of errors for 1000 instances of the noise, for a fixed \code{grains} phantom. }
		\label{fig:error_histogram}
	\end{figure}
	
	To provide further support for the quality of our methods, we ran the proposed algorithms 100 times
	on each of seven phantoms available from \code{phantomgallery}.
	Additionally, for each run we assign points based on which method produces the best result.
	Our point system assigns a score of 1 point to the method with the best reconstruction, half a point for the second best, and no point for the worst.
	Hence, a score of 100 means the algorithm produced the best result each time, while a score of 0 means it produced the worst result each time. This allows us to see roughly how well the methods behave relative to each other.
	The final tally is contained in Table~\ref{tab:phantom_gallery_table}.

	The \MSA\ is probably the best algorithm
	under consideration when it comes to producing a high-quality image. For 4 out of 7 phantoms it produces the best image in the vast majority of tests. The \TA\ also produces good reconstructions, being the best for two phantoms. Combined, the \TA\ and the \MSA\ produce the best image in the majority of test cases for 6 out of 7 phantoms. Interestingly, only for the \code{binary} phantom does \KO\ have the highest score.
	
	To provide more insight, we also display the average errors for each particular phantom
	and the total average error in the first three columns of Table~\ref{tab:phantom_gallery_table}.
	Here we see another remarkable aspect
	of the proposed algorithms that was also observed in Figure~\ref{fig:error_histogram}: averaged over all phantoms, both proposed algorithms are at least as good as \KO; the \MSA\ is significantly better.

	\begin{table}
		\begin{small}
			\begin{center}
				\caption{Average relative errors, work units and score
					of 100 instances for each of seven phantoms. TA is the \TA, MSA is the \MSA\ and KO is the \KO\ algorithm.}\label{tab:phantom_gallery_table}
				\begin{tabular}{l|rrr|rrr|rrr} \hline
					& \multicolumn{3}{c|}{Relative errors} & \multicolumn{3}{c|}{Work units} & \multicolumn{3}{c}{Score} \\
					Phantom & TA & MSA & KO & TA & MSA & KO & TA & MSA & KO \\ \hline \rule{0pt}{2.3ex}
					\code{shepplogan} & 0.166 & 0.175 & 0.169 & 36.6 & 16.0 & 20.7 & 74.5 & 14.0 & 61.5 \\ 
					\code{smooth} & 0.194 & 0.105 & 0.163 & 28.9 & 17.1 & 17.4 & 12.5 & 99.0 & 38.5 \\ 
					\code{binary} & 0.215 & 0.222 & 0.209 & 26.5 & 15.6 & 23.7 & 54.5 & 12.0 & 83.5 \\ 
					\code{threephases} & 0.147 & 0.140 & 0.156 & 35.1 & 15.8 & 10.9 & 47.0 & 88.5 & 14.5 \\ 
					\code{threephasessmooth} & 0.132 & 0.110 & 0.142 & 34.0 & 17.2 & 10.4 & 32.0 & 100 & 18.0 \\ 
					\code{fourphases} & 0.190 & 0.202 & 0.193 & 38.0 & 16.2 & 20.3 & 81.5 & 7.5 & 61.0 \\ 
					\code{grains} & 0.134 & 0.092 & 0.149 & 40.1 & 16.4 & 15.3 & 41.5 & 100 & 8.5 \\[1mm]  \hline 
					Average & 0.168 & 0.149 & 0.169 & 34.2 & 16.3 & 17.0 & 49.1 & 60.1 & 40.8 \\ \hline 
				\end{tabular}
			\end{center}
		\end{small}
	\end{table}

	Finally, in the middle three columns of Table~\ref{tab:phantom_gallery_table}
	we also compare the methods when it comes to computing cost expressed in our work units
	(recall that one work unit is the work required to complete one sweep of Kaczmarz's method).
	We should point out that \KO\ has a very low work load, which is due to the fact
	that consulting the oracle is assumed to be free. The corresponding workload
	should therefore be read as a lower bound on the work required for Kaczmarz's method equipped with any pure stopping rule. Interestingly, the \MSA\
	is not too far off, with less work required when averaged over all 7 phantoms.
	The \TA\ usually requires much more work than the oracle, which is to be expected of course: if the \TA\ would stop at the same point as the oracle, it would have done roughly twice the work plus the slack.

	\section{Conclusion and future work}
	
	We presented a new approach to noisy CT reconstruction based on Kaczmarz's method.
	The regularizing property of the proposed methods stems from the semi-convergence of Kaczmarz's method, where the
	iteration number is used as the regularization parameter.
	The problem of choosing the regularization parameter takes the form of a stopping rule, typically
	based on a statistical analysis of the noise and the error.
	
	Our key idea is to combine two Kaczmarz iterations with different row orderings and the same convergence
	rate, which allows us to compute an \emph{error gauge}, i.e., an estimate of the reconstruction error.
	We then stop the iterations when the error gauge is minimum, providing an alternative to existing stopping rules and avoiding assumptions about the noise statistics.
	
	When the original linear system \eqref{eq:Axb} is consistent, we can prove rigorously that the error gauge estimates the reconstruction error.
	For noisy systems we argued that the error gauge represents semi-convergence with a reasonable fidelity.
	
	We suggested two algorithms that utilize the error gauge:\ the \TA\ and the \MSA.
	The former uses the error gauge directly as a stopping rule, stopping when the error gauge is minimal.
	The latter uses the error gauge to determine approximately optimal step sizes
	for every iteration.
	We showed that the \MSA\ converges monotonically to a locally optimal pair of approximations.
	It therefore converges to an approximation of the semi-convergence point, precluding the need for a stopping rule.
	
	Using several numerical experiments from parallel-beam X-ray CT, we demonstrated that the proposed algorithms perform very well indeed.
	As a reference, we used an oracle for the standard Kaczmarz algorithm that provides the exact error, which is therefore able to pick
	the best possible reconstruction from the sequence of iterates. Our proposed algorithms -- whose output is the average of the up- and down-sweeps -- perform on average better than Kaczmarz's method equipped with the oracle.
	
	For four out of the seven phantoms that we consider, the \MSA\ produce the best image for a large majority of cases.
	For two other phantoms, the \TA\ produces the best image in the majority of the test cases.
	Only for a single phantom our algorithms fail to produce the best reconstruction in the majority of the cases. Clearly, the \MSA\ is a solid choice for a reconstruction algorithm. However, it does require additional storage. Depending on what method is employed to detect a minimum in the error gauge, the \TA\ has the same memory requirements as the statistical stopping rules. Therefore, if memory is a concern, the \TA\ is a good alternative.
	
	\subsection*{Acknowledgements}
	We would like to thank the referees for their excellent comments and suggestions.
	B.S.~van Lith is supported by
	the EuroTech Postdoc Programme, co-funded by the European Commission under its framework programme Horizon 2020. Grant Agreement number 754462.
	
	\bibliographystyle{siam}

\begin{thebibliography}{10}
		
		\bibitem{bardsley}
		{\sc J.~M. Bardsley}, {\em Applications of a nonnegatively constrained
			iterative method with statistically based stopping rules to {CT}, {PET}, and
			{SPECT} imaging}, Electron. Trans. Numer. Anal., 38 (2011), pp.~34--43.
		
		\bibitem{bertero}
		{\sc M.~Bertero and P.~Boccacci}, {\em Introduction to {I}nverse {P}roblems in
			{I}maging}, CRC Press, 1998.
		
		\bibitem{bjorck_elfving}
		{\sc {\AA}.~Bj\"orck and T.~Elfving}, {\em Accelerated projection methods for
			computing pseudoinverse solutions of systems of linear equations}, BIT Numer.
		Math., 19 (1979), pp.~145--163.
		
		\bibitem{shapiro}
		{\sc J.~F. Bonnans and A.~Shapiro}, {\em Perturbation {A}nalysis of
			{O}ptimization {P}roblems}, Springer, Berlin Heidelberg, 2000.
		
		\bibitem{boyd}
		{\sc S.~P. Boyd and L.~Vandenberghe}, {\em Convex Optimization}, Cambridge
		University Press, 2004.
		
		\bibitem{elfving_2012}
		{\sc T.~Elfving, P.~C. Hansen, and T.~Nikazad}, {\em Semiconvergence and
			relaxation parameters for projected {SIRT} algorithms}, SIAM J. Sci. Comput.,
		34 (2012), pp.~A2000--A2017.
		
		\bibitem{elfving_2014}
		\leavevmode\vrule height 2pt depth -1.6pt width 23pt, {\em Semi-convergence
			properties of {K}aczmarz's method}, Inverse Problems, 30 (2014), p.~055007.
		
		\bibitem{elfving}
		{\sc T.~Elfving and T.~Nikazad}, {\em Properties of a class of block-iterative
			methods}, Inverse Problems, 25 (2009), p.~115011.
		
		\bibitem{Girard89}
		{\sc D.~A. Girard}, {\em A fast `{M}onte {C}arlo' cross-validation procedure
			for least squares problems with noisy data}, Numer. Math., 56 (1989),
		pp.~1--23.
		
		\bibitem{gordon}
		{\sc D.~Gordon}, {\em Parallel {ART} for image reconstruction in {CT} using
			processor arrays}, International Journal of Parallel, Emergent and
		Distributed Systems, 21 (2006), pp.~365--380.
		
		\bibitem{gordon_bender}
		{\sc R.~Gordon, R.~Bender, and G.~T. Herman}, {\em Algebraic reconstruction
			techniques {(ART)} for three-dimensional electron microscopy and {X}-ray
			photography}, J. Theor. Biol., 29 (1970), pp.~471--481.
		
		\bibitem{discretization}
		{\sc K.~Hahn, H.~Sch{\"o}ndube, K.~Stierstorfer, J.~Hornegger, and F.~Noo},
		{\em A comparison of linear interpolation models for iterative {CT}
			reconstruction}, Medical Physics, 43 (2016), pp.~6455--6473.
		
		\bibitem{Hansen1998}
		{\sc P.~C. Hansen}, {\em Rank-Deficient and Discrete Ill-Posed Problems:\
			Numerical Aspects of Linear Inversion}, SIAM, Philadelphia, 1998.
		
		\bibitem{Hansen2018}
		{\sc P.~C. Hansen and J.~S. J{\o}rgensen}, {\em {AIR Tools II}: algebraic
			iterative reconstruction methods, improved implementation}, Numerical
		Algorithms, 79 (2018), pp.~107--137.
		
		\bibitem{horn_johnson}
		{\sc R.~A. Horn and C.~R. Johnson}, {\em Matrix Analysis}, Cambridge University
		Press, 2012.
		
		\bibitem{knuth}
		{\sc D.~E. Knuth}, {\em The Art of Computer Programming}, Addison-Wesley,
		Redwood City, 1997.
		
		\bibitem{natterer}
		{\sc F.~Natterer}, {\em The Mathematics of Computerized Tomography}, {SIAM,
			Philadelphia}, 2001.
		
		\bibitem{papadimitriou}
		{\sc C.~H. Papadimitriou}, {\em Computational {C}omplexity}, Addison-Wesley,
		Redwood City, 1994.
		
		\bibitem{Popa2018}
		{\sc C.~Popa}, {\em Convergence rates for {K}aczmarz-type algorithms},
		Numerical Algorithms, 79 (2018), pp.~1--17.
		
		\bibitem{Reichel2013}
		{\sc L.~Reichel and G.~Rodriguez}, {\em Old and new parameter choice rules for
			discrete ill-posed problems}, Numerical Algorithms, 63 (2013), pp.~65--87.
		
		\bibitem{rencher}
		{\sc A.~C. Rencher and G.~B. Schaalje}, {\em Linear Models in Statistics},
		Wiley, Hoboken, 2008.
		
		\bibitem{saad}
		{\sc Y.~Saad}, {\em Iterative {M}ethods for {S}parse {L}inear {S}ystems}, SIAM,
		Philadelphia, 2003.
		
		\bibitem{SantosPierro03}
		{\sc R.~J. Santos and A.~R.~D. Pierro}, {\em A cheaper way to compute
			generalized cross-validation as a stopping rule for linear stationary
			methods}, J. Comp. Graphical Statistics, 12 (2003), pp.~417--433.
		
		\bibitem{sorensen}
		{\sc H.~H.~B. S{\o}rensen and P.~C. Hansen}, {\em Multicore performance of
			block algebraic iterative reconstruction methods}, SIAM J. Sci. Comput., 36
		(2014), pp.~C524--C546.
		
		\bibitem{stewart}
		{\sc G.~W. Stewart}, {\em Matrix Algorithms Volume {II}: Eigensystems}, {SIAM,
			Philadelphia}, 2001.
		
		\bibitem{tanabe}
		{\sc K.~Tanabe}, {\em Projection method for solving a singular system of linear
			equations and its applications}, Numer. Math., 17 (1971), pp.~203--214.
		
		\bibitem{turchin}
		{\sc V.~Turchin}, {\em Solution of the {F}redholm equation of the first kind in
			a statistical ensemble of smooth functions}, USSR Computational Mathematics
		and Mathematical Physics, 7 (1967), pp.~79--96.
		
		\bibitem{van_dijke}
		{\sc M.~C.~A. van Dijke, H.~A. van~der Vorst, and M.~A. Viergever}, {\em On the
			relation between {ART}, block-{ART} and {SIRT}}, in Medical Images:
		Formation, Handling and Evaluation, A.~E. Todd-Pokropek and M.~A. Viergever,
		eds., Springer, Berlin Heidelberg, 1992, pp.~377--396.
		
		\bibitem{vogel}
		{\sc C.~Vogel}, {\em Computational Methods for Inverse Problems}, {SIAM,
			Philadelphia}, 2002.
		
		\bibitem{wahba}
		{\sc G.~Wahba}, {\em Spline Models for Observational Data}, {SIAM,
			Philadelphia}, 1990.
		
	\end{thebibliography}

\end{document}